\documentclass[reqno]{amsart}

\usepackage[utf8]{inputenc}
\usepackage{amssymb, amsmath, amsthm, xcolor, enumerate, makecell, marginnote, mathabx, mathrsfs}
\usepackage{bbm}
\usepackage{hyperref}

\usepackage{esint} 

\numberwithin{equation}{section}
\theoremstyle{plain}
\newtheorem{lemma}{Lemma}[section]
\newtheorem{theorem}[lemma]{Theorem}

\newtheorem{conjecture}[lemma]{Conjecture}

\theoremstyle{definition}
\newtheorem{definition}[lemma]{Definition}
\newtheorem{example}[lemma]{Example}
\newtheorem{remark}[lemma]{Remark}

\newenvironment{lemmabis}[1]
  {\renewcommand{\thelemma}{\ref{#1}$^*$}%
   \addtocounter{lemma}{-1}%
   \begin{lemma}}
  {\end{lemma}}

\newcommand{\C}{\mathbb{C}}

\newcommand{\N}{\mathbb{N}}

\newcommand{\bbP}{\mathbb{P}}
\newcommand{\Q}{\mathbb{Q}}
\newcommand{\R}{\mathbb{R}}
\newcommand{\bbS}{\mathbb{S}}

\newcommand{\Z}{\mathbb{Z}}

\newcommand{\cD}{\mathcal{D}}
\newcommand{\cE}{\mathcal{E}}
\newcommand{\cF}{\mathcal{F}}

\newcommand{\cI}{\mathcal{I}}
\newcommand{\cJ}{\mathcal{J}}

\newcommand{\cM}{\mathcal{M}}
\newcommand{\cN}{\mathcal{N}}

\newcommand{\cW}{\mathcal{W}}

\newcommand{\tI}{\mathrm{I}}
\newcommand{\tII}{\mathrm{II}}

\newcommand{\tO}{\mathrm{O}}

\newcommand{\fka}{\mathfrak{a}}

\newcommand{\fke}{\mathfrak{e}}

\newcommand{\fkA}{\mathfrak{A}}

\newcommand{\fkI}{\mathfrak{I}}
\newcommand{\fkJ}{\mathfrak{J}}

\newcommand{\fkN}{\mathfrak{N}}

\newcommand{\fkR}{\mathfrak{R}}

\newcommand{\fkX}{\mathfrak{X}}

\newcommand{\bdA}{\mathbf{A}}
\newcommand{\bdB}{\mathbf{B}}
\newcommand{\bdC}{\mathbf{C}}
\newcommand{\bdD}{\mathbf{D}}

\newcommand{\bdI}{\mathbf{I}}
\newcommand{\bdJ}{\mathbf{J}}

\newcommand{\bda}{\mathbf{a}}
\newcommand{\ba}{\mathbf{a}}
\newcommand{\bdb}{\mathbf{b}}

\newcommand{\bdi}{\mathbf{i}}
\newcommand{\bdj}{\mathbf{j}}
\newcommand{\bdk}{\mathbf{k}}
\newcommand{\bk}{\mathbf{k}}

\newcommand{\bdp}{\mathbf{p}}

\newcommand{\bds}{\mathbf{s}}
\newcommand{\bdt}{\mathbf{t}}
\newcommand{\bdu}{\mathbf{u}}

\newcommand{\bdw}{\mathbf{w}}
\newcommand{\bdx}{\mathbf{x}}
\newcommand{\bx}{\mathbf{x}}
\newcommand{\bdy}{\mathbf{y}}

\newcommand{\bdz}{\mathbf{z}}

\newcommand{\bdzero}{\mathbf{0}}
\newcommand{\bzero}{\mathbf{0}}

\newcommand{\bdalpha}{\boldsymbol{\alpha}}

\newcommand{\bdbeta}{\boldsymbol{\beta}}

\newcommand{\bddelta}{\boldsymbol{\delta}}
\newcommand{\bdel}{\boldsymbol{\delta}}

\newcommand{\bdzeta}{\boldsymbol{\zeta}}
\newcommand{\bdeta}{\boldsymbol{\eta}}
\newcommand{\bdtheta}{\boldsymbol{\theta}}

\newcommand{\ud}{\mathrm{d}}

\newcommand{\floor}[1]{\lfloor #1 \rfloor }
\newcommand{\ceil}[1]{\lceil #1 \rceil }
\newcommand{\supp}{\mathrm{supp}\,}

\def\delp{{ \bdel^{\times}}}

\def\lgQ*{{\lceil\frac{\log 4Q^*}{\log 2}\rceil}}

\begin{document}

\title[Rational points near manifolds]{Counting rational points near manifolds: a refined estimate, a conjecture and a variant}

\date{}

\author[J. Hickman]{ Jonathan Hickman }
\address{School of Mathematics and Maxwell Institute for Mathematical Sciences, James Clerk Maxwell Building, The King's Buildings, Peter Guthrie Tait Road, Edinburgh, EH9 3FD, UK.}
\email{jonathan.hickman@ed.ac.uk}

\author[R. Srivastava]{ Rajula Srivastava }
\address{Rajula Srivastava; 
\newline 
Department of Mathematics, 
University of Wisconsin 480 Lincoln Drive, 
Madison, WI, 53706, USA}
\email{
rsrivastava9@wisc.edu}

\author[J. Wright]{James Wright}
\address{James Wright: School of Mathematics and Maxwell Institute for Mathematical Sciences, James Clerk Maxwell Building, The King's Buildings, Peter Guthrie Tait Road, Edinburgh, EH9 3FD, UK.}
\email{j.r.wright@ed.ac.uk}

\begin{abstract} Refining an argument of the second author, we improve the known bounds for the number of rational points near a submanifold of $\R^d$ of intermediate dimension under a natural curvature condition. Furthermore, in the codimension $2$ case we formulate a conjecture concerning this count. The conjecture is motivated in part by interpreting certain codimension $2$ submanifolds of $\R^{2m+2}$ as complex hypersurfaces in $\C^{m+1}$ and using the complex structure to provide a natural reformulation of the curvature condition. Finally, we provide further evidence for the conjecture by proving a natural variant for $n \geq 2$ in which rationals are replaced with Gaussian rationals.
\end{abstract}

\maketitle




\section{Introduction}




\subsection{The problem} 
For $n$, $R \in \N$ and $d := n + R$, let $\cM \subseteq \R^d$ be a bounded, immersed $n$-dimensional smooth submanifold of $\R^d$ with boundary. For $Q \in \N$ and $\delta \in (0, 1/2)$, we define the counting function 
\begin{equation*}
    N_{\cM}(Q, \delta) := \#\big\{ (\bdp, q) \in \Z^d : 1 \leq q \leq Q, \, \mathrm{dist}(\cM, \bdp/q) \leq \delta/q \big\},
\end{equation*}
where the distance function is given by
\begin{equation*}
  \mathrm{dist}(\cM, \bdp/q) := \inf_{\bdx \in \cM} |\bdx - \bdp/q|_{\infty}.   
\end{equation*}
An interesting avenue of research is to obtain upper and lower bounds for $N_{\cM}(Q, \delta)$, or indeed asymptotics, for $\cM$ belonging to a suitable class. This problem has attracted considerable interest in recent years and has various applications to dimension growth-type problems and diophantine approximation on manifolds: see, for instance, \cite{Huang2020, SY2022, Srivastava2025} and references therein. 

If $\cM \subseteq \R^d$ has dimension $n$ as above, then it is conjectured that, for all $\nu > 0$, 
the bound
\footnote{Here, and throughout the rest of the article, given a list of objects $L$ and non-negative real numbers $A$, $B$, we write $A \ll_L B$ or $B \gg_L A$ or $A = O_L(B)$ if $A \leq C_L B$ where $C_L > 0$ is a constant depending only on the objects listed in $L$, and possibly a choice of dimension $n$ and underlying submanifold $\cM$. We write $A \sim_L B$ when both $A \ll_L B$ and $B \gg_L A$.}
\begin{equation}\label{eq: Huang conj}
    N_{\cM}(Q, \delta) \ll_{\nu} \delta^R Q^{n+1} + Q^{n + \nu}
\end{equation}
holds under `proper curvature conditions' on $\cM$. More precisely, the above bound would be a direct consequence of a conjecture of J.J. Huang (see \cite[Conjecture 3.1]{huang2024extremal}) for nondegenerate manifolds (that is, manifolds not contained in any proper affine subspace of $\R^d$). Simple examples show that some curvature hypothesis is necessary: for instance, the bound \eqref{eq: Huang conj} fails for a compact piece of a rational $n$-plane. Determining sufficient geometric conditions under which \eqref{eq: Huang conj} holds for a curved manifold $\cM$ for general $n$ and $R$ is an interesting and largely open problem. For the hypersurface case, where $R = 1$, under the natural condition that $\cM$ has nonvanishing Gaussian curvature, the works of Huxley~\cite{Huxley1994} and Vaughan--Velani~\cite{VV2006}, and then J.J. Huang~\cite{Huang2020} established that
\eqref{eq: Huang conj} indeed holds in dimensions two and greater than two respectively. More recently, Srivastava--Technau \cite{st2023} and M. Chen \cite{chen2025rational} have established sharp estimates for rational points near homogeneous hypersurfaces and planar curves of finite type respectively, when the Gaussian curvature vanishes at isolated points. 
\medskip

In \cite{Huang2020}, J.J. Huang used an inductive argument relying on a combination of Poisson summation, projective duality and stationary phase to establish the sharp rational point count near hypersufaces with non-vanishing Gaussian curvature. This Fourier analytic approach was developed further by Schindler--Yamagishi \cite{SY2022} and later the second author \cite{Srivastava2025} to treat certain manifolds of codimension greater than one. For such intermediate dimensional manifolds, Schindler--Yamagishi \cite{SY2022} introduced the following curvature condition.

\begin{definition}\label{dfn: CC} Given an open set $U \subseteq \R^n$ and a smooth function  $f \colon U \to \R^R$, we say $f = (f_1, \ldots, f_R)$ satisfies the \textit{curvature condition} if
\begin{equation}
\label{eq: CC}
 \det\Big(\sum_{\ell=1}^R \theta_{\ell}  f_{\ell}''(\bdx)\Big) \neq 0 \qquad \textrm{for all $\bdx \in U$, $(\theta_1, \dots, \theta_R) \in \bbS^{R-1}$.} \tag{\textrm{CC}}   
\end{equation}
Here $g''(\bdx)$ denotes the second derivative of a smooth map $g \colon \R^n \to \R$ at $\bdx \in U$, viewed as an $n \times n$ matrix of mixed partials. 
\end{definition}

We shall say an $n$-dimensional submanifold $\cM$ of $\R^{n + R}$ satisfies the \textit{curvature condition} (CC) if $\cM$ is given by a graph
\begin{equation}\label{eq: real graph}
  \{(\bdx, f(\bdx)) : \bdx \in X \} \subset \R^{n + R}  
\end{equation}
for an open, pre-compact set $X \Subset U$ and a smooth function $f \colon U \to \R^R$ satisfying \eqref{eq: CC}. This definition is slightly unsatisfactory, since it relies on working with a particular choice of parametrisation for $\cM$, but we shall present an intrinsic formulation in \S\ref{subsec: implicit CC}. 
In the hypersurface case, where $R = 1$, the manifold is parametrised by a single graphing function $f \colon U \to \R$ and \eqref{eq: CC} reduces to the condition that the Hessian $\det  f''(\bdx)$ is nonvanishing for all $\bdx \in U$. This is precisely the statement that $\cM$ has nonvanishing Gaussian curvature, expressed in graph coordinates.\medskip

Condition (CC) turns out to be very strong and rigid. In particular, there are relatively few dimension and codimension pairs $(n, R)$ for which submanifolds $\cM$ satisfying (CC) actually exist: see, for instance, \cite[Remark 1.14]{Srivastava2025} for further details. The condition \eqref{eq: CC} also features extensively in the harmonic analysis literature and is closely related to the \textit{rotational curvature condition} of Phong--Stein: see \cite{Gressman2015, rss2022lebesgue} and references therein.

\begin{definition}  We say a pair of dimensions $(n, R) \in \N$ is \textit{(CC)-admissible} if there exists a codimension $R$ submanifold of $\R^{n+R}$ satisfying (CC). For such a pair $(n,R)$, we define $\fke(n,R)$ to be the supremum over all exponents $e \in \R$ for which, given any codimension $R$ submanifold $\cM$ of $\R^{n+R}$ satisfying (CC), the bound
\begin{equation*}
    N_{\cM}(Q, \delta) \ll_e \delta^R Q^{n+1} + Q^{n + 1 - e}
\end{equation*}
holds for all $Q \in \N$ and $\delta \in (0,1/4)$.    
\end{definition}

A trivial count shows that $\fke(n,R) \geq 0$, whilst Huang's conjecture \eqref{eq: Huang conj} corresponds to $\fke(n,R) \geq 1$. We summarise various results bounding $\fke(n,R)$ in Figure~\ref{fig: progress}.\medskip

\begin{figure}
    \centering
     \begin{tabular}{l|c|c}
      Result   & $\fke(n,R) \geq$ & (CC)-admissible $(n, R) \in \N^2$   \\
      \Xhline{2\arrayrulewidth}
     Trivial  & $0$ & All  \\
     \hline
     Huxley~\cite{Huxley1994}, Vaughan--Velani~\cite{VV2006}  & $1$ & $(1, 1)$ \\
     \hline
     Huang~\cite{Huang2020}  & $1$ & $(n, 1)$ for $n \geq 1$ \\
     \hline
     Schindler--Yamagishi \cite{SY2022}  & $ \frac{nR}{n + 2(R-1)}$ & All  \\[2pt]
     \hline
     Srivastava \cite{Srivastava2025} & $ \frac{(n+2)R}{n + 2R}$ & $(n, 1)$ or $(n,2)$ for $n \geq 1$ \\ 
     \hline
     Srivastava \cite{Srivastava2025} & $ \frac{n^2R}{n^2 + 2(R-1)n - 4}$ & $(n, R)$ for $n \geq 1$, $R \geq 3$.
    \end{tabular}
    \caption{Progress towards the rational point counting problem for neighbourhoods of submanifolds satisfying (CC).}
    \label{fig: progress}
\end{figure}

In particular, observe that for $n \geq 2$ and $R \geq 2$, the exponent $\frac{nR}{n + 2(R-1)}$ of Schindler--Yamagishi~\cite{SY2022}, featured in Figure~\ref{fig: progress}, is larger (and therefore better) than the exponent $1$ posited by Huang's conjecture \eqref{eq: Huang conj}. Moreover, exponents $\frac{(n+2)R}{n + 2R}$ and $\frac{n^2R}{n^2 + 2(R-1)n - 4}$ from Srivastava \cite{Srivastava2025} are larger still. Thus, under the curvature condition (CC), significantly stronger bounds hold than those conjectured in \eqref{eq: Huang conj}.




\subsection{A refined estimate} Our first result is a refinement of the rational point count from \cite{Srivastava2025}. In particular, we extend the numerology of \cite{Srivastava2025} for codimension $1$ and $2$ submanifolds (see Figure~\ref{fig: progress}) to submanifolds of higher codimension. 

\begin{theorem}\label{thm: refined count} For all (CC)-admissible $(n, R) \in \N^2$, we have $\fke(n, R) \geq \frac{(n+2)R}{n + 2R}$.     
\end{theorem}

Comparing the numerology of Theorem~\ref{thm: refined count} with that of \cite{Srivastava2025}, as tabulated in Figure~\ref{fig: progress}, we see that the former provides a modest improvement over the latter whenever $R \geq 3$. To achieve this improvement, we exploit additional cancellation in error terms arising from stationary phase asymptotics used repeatedly in \cite{Srivastava2025}. We remark that the inefficiencies in the error term estimation were already highlighted in \cite{Srivastava2025}, where Theorem~\ref{thm: refined count} was also conjectured to hold.

As in \cite{Srivastava2025}, we shall in fact establish a more general version of Theorem~\ref{thm: refined count} for the number of rational points with bounded denominators contained in a non-isotropic neighbourhood of $\cM$: see Theorem~\ref{thm: aniso refined} below.

The proof of Theorem~\ref{thm: refined count} is rather technical in nature and relies on targeted improvements to specific estimates in \cite{Srivastava2025}. Nevertheless, we feel that the result is of interest since it represents the apparent theoretical limit of Fourier analytic methods based on the induction argument in \cite{Huang2020}, at least when applied to surfaces satisfying (CC). 




\subsection{A conjecture: the codimension 2 case}\label{subsec: codim 2 conj} The work of Schindler--Yamagishi \cite{SY2022} and, to a further extent, \cite{Srivastava2025} and Theorem~\ref{thm: refined count} above, demonstrate that, under the condition (CC), it is possible to go beyond Huang's conjectural bounds \eqref{eq: Huang conj} for codimension $R \geq 2$. It is therefore natural to ask what the sharp range of estimates should be for this class of submanifolds. Here we focus on the codimension $2$ case.

\begin{conjecture}\label{conj: codim 2} For all sufficiently large even $n \in \N$, we have $\fke(n, 2) = 2$.
\end{conjecture}

Using explicit examples, we shall show below that $(n, 2)$ is (CC) admissible for all $n \in \N$ even, and $\fke(n, 2) \leq 2$ for $n \geq 8$. This partially motivates the numerology of the conjecture. On the other hand, $(n, 2)$ cannot be (CC) admissible if $n \in \N$ is odd,\footnote{To see this, note that when $n \in \N$ is odd the determinant in (CC) is an odd function of $\bdtheta$ and therefore must vanish on the unit sphere $\bbS^{R-1}$ when $R \geq 2$.} and so there is no loss of generality in only considering manifolds of even dimension. As discussed in \S\ref{subsec: n=2} below, $\fke(2, 2) = 2$ fails, which necessitates the restriction to large $n$. To date, the best partial result toward Conjecture~\ref{conj: codim 2} is the bound 
\begin{equation*}
    \fke(2m, 2) \geq 1 + \frac{m}{m+2} \qquad \textrm{for all $m \in \N$}
\end{equation*}
from \cite{Srivastava2025}. Here we do not improve the estimates beyond this, but provide supporting evidence for the numerology in Conjecture \ref{conj: codim 2}.

The motivation for Conjecture~\ref{conj: codim 2} comes from reinterpreting the condition \eqref{eq: CC} in terms of complex derivatives. Given $U \subset \C^m$ an open set and $f \colon U \to \C$ a holomorphic function, consider the analytic hypersurface 
\begin{equation}\label{eq: analytic M}
   \cM = \{(\bdz, f(\bdz)) : \bdz \in Z \} \subseteq \C^{m+1}
\end{equation}
for an open, pre-compact set $Z \Subset U$, where  $\bdz = (z_1, \ldots, z_m) \in \C^m$ are complex coordinates. Note that we may also view $\cM$ as a real codimension $2$ submanifold $\cM_{\R}$ of $\R^{2m+2}$. Indeed, we may write $f$ in terms of its real and imaginary parts as
\begin{equation}\label{eq: f real img}
    f(\bdx + i\bdy)  = u(\bdx,\bdy) + i v(\bdx,\bdy) \quad \textrm{for all $(\bdx, \bdy) \in U_{\R}$}
\end{equation}
where 
\begin{equation*}
    A_{\R} := \{(\bdx, \bdy) \in \R^m \times \R^m : \bdx + i \bdy \in A\} \qquad \textrm{for $A \subseteq \C^m$}
\end{equation*}
and $u$, $v \colon U_{\R} \to \R$ are real valued functions satisfying the Cauchy--Riemann equations. We then define $\cM_{\R}$ as the graph of $f_{\R} := (u, v) \colon U_{\R} \to \R^2$ over $Z_{\R}$. 

Using the Cauchy--Riemann equations, one may easily relate the condition \eqref{eq: CC} for the real function $f_{\R} \colon U_{\R} \to \R^2$ to the complex derivatives of the holomorphic function $f \colon U \to \C$. In particular, in \S\ref{sec: examples} we discuss the following identity. 

\begin{lemma}\label{lem: ND vs CND} Let $U \subseteq \C^m$ be open and $f \colon U \to \C$ be a holomorphic function satisfying \eqref{eq: f real img} as above. Then
    \begin{equation}
\label{eq: ND vs CND}
 \det \bigl(\theta_1 u''(\bdx, \bdy) + \theta_2 v''(\bdx, \bdy) \bigr)  =
(-1)^m \, (\theta_1^2 + \theta_2^2)^m | \det f''(\bdx + i\bdy)|^2
\end{equation}
holds for all $\theta_1$, $\theta_2 \in \R$.
\end{lemma}
We emphasise that the second derivatives $u''$ and $v''$ on the left-hand side of \eqref{eq: ND vs CND} are $m \times m$ real matrices defined with respect to the real differential structure, whereas $f''$ is an $m \times m$ complex matrix defined with respect to the complex differential structure. In light of the formula \eqref{eq: ND vs CND}, we make the following definition. 

\begin{definition} Given $U \subset \C^m$ an open set and $f \colon U \to \C$ a holomorphic function, we say $f$ satisfies the \textit{complex curvature condition} if
\begin{equation*}
\det  f''(\bdz)  \neq 0 \qquad \textrm{for all $\bdz \in U$.} \eqno{\textrm{($\C$-CC)}}   
\end{equation*}
\end{definition}

As in the real case, we say an analytic hypersurface $\cM \subseteq \C^{m+1}$ satisfies the \textit{complex curvature condition} ($\C$-CC) if it is of the form \eqref{eq: analytic M} for $f \colon U \to \C$ a holomorphic function satisfying ($\C$-CC). 

By Lemma~\ref{lem: ND vs CND}, the condition (CC) reduces to ($\C$-CC) for the class of codimension $2$ surfaces in $\R^{2m+2}$ arising from analytic hypersurfaces in $\C^{m+1}$. Furthermore, by the remarks following Definition~\ref{dfn: CC}, the condition ($\C$-CC) is the natural complex analogue of the nonvanishing Gaussian curvature condition for real hypersurfaces from \cite{Huang2020}. These observations provide a natural family of (real) co-dimension 2 manifolds satisfying (CC), essentially by `complexifying' the examples which demonstrate the sharpness of Huang's result for real hypersurfaces. 

\begin{example}[Complex sphere]\label{ex: sphere} Let $B_r = B(0,r) \subset \C^m$ be an open ball around the origin of radius $0 < r < 1$ and consider the portion of the complex unit sphere
\begin{equation*}
    \bbS_{\C}^m := \bigl\{ (\bdz, z_{m+1}) \in B_r \times \C : F(\bdz, z_{m+1}) = 1  \bigr\},
\end{equation*}
where $F \colon \C^{m+1} \to \C$ is the homogeneous quadratic form
\begin{equation*}
    F(z_1, \dots, z_{m+1}) := z_1^2 + \cdots + z_{m+1}^2.
\end{equation*}

By the holomorphic implicit function theorem, provided $r$ is chosen sufficiently small, $\bbS_{\C}^m$ admits a parametrisation of the form \eqref{eq: analytic M} where $f \colon B_r \to \C$ satisfies ($\C$-CC). We identify $\bbS_{\C}^m$ with a codimension $2$ real submanifold of $\R^{2m+2}$. Using Lemma~\ref{lem: ND vs CND}, one can show this real submanifold satisfies (CC); we defer the details to \S\ref{sec: examples}. Furthermore, under this identification,
\begin{equation*}
    N_{\bbS^m_{\C}}(Q,0) = \#\big\{(\bda,q) \in \Z[i]^{m+1} \times \N : F(\bda) = q^2 ,\, \bda/q \in B_r \times \mathbb{C} \textrm{ and } 1 \leq q \leq Q \big\}.
\end{equation*}

To estimate $N_{\bbS^m_{\C}}(Q,0)$, we appeal to classical work of Siegel \cite{Siegel1922, Siegel1944} which extends the circle method bounds on the Waring problem to the setting of general algebraic number fields. Here we are interested in the field of Gaussian rationals $\Q(i)$. Given\footnote{The main theorem in \cite{Siegel1944} has the additional restriction that $\nu$ is \textit{totally positive}. However, this condition vacuously holds for all $\nu \in \Q(i)$ since $\Q(i)$ is a totally imaginary field: see, for instance, the first line of \cite{Siegel1921}.} $\nu \in \Z[i]$ in the ring of integers of $\Q(i)$, define
\begin{equation*}
  A_m(\nu) := \#\big\{ (\lambda_1, \dots, \lambda_m) \in \Z[i]^m : |\lambda_j|^2 \leq r|\nu|,\, 1 \leq j \leq m \textrm{ and } \lambda_1^2 + \cdots + \lambda_m^2 = \nu \big\}.
\end{equation*}
Suppose $\nu = a + bi$ with $a$, $b \in \Z$ and $b$ even. For $m \geq 5$, the work of Siegel \cite{Siegel1922, Siegel1944} gives the asymptotic
\begin{equation}\label{eq: Siegel asympt}
A_m(\nu) = C_m(\nu)|\nu|^{m-2} + o_m(|\nu|^{m-2}),
\end{equation}
where $C_m(\nu) \sim_m 1$.\footnote{On the other hand, if $\nu = a + bi$  with $b$ odd, then it is trivial to see that $A_m(\nu) = 0$.} Thus, for $m \geq 4$, we have
\begin{equation*}
    N_{\bbS_{\C}^m}(Q, 0) = \sum_{q = 1}^Q A_{m+1}(q^2) \gtrsim Q^{2m-1},
\end{equation*}
provided $Q$ is sufficiently large, depending on $m$.
\end{example}

Unpacking the definitions, Conjecture~\ref{conj: codim 2} states the following. For any $m \in \N$ and codimension $2$ submanifold $\cM$ of $\R^{2m+2}$ satisfying (CC), for all $\varepsilon > 0$, the bound
\begin{equation*}
    N_{\cM}(Q, \delta) \ll_{\varepsilon} \delta^2 Q^{2m+1} + Q^{2m - 1 + \varepsilon}
\end{equation*}
holds for all $Q \in \N$ and $\delta \in (0,1/4)$. Example~\ref{ex: sphere} shows that, if true, this would be sharp in the sense that, for all $m \geq 4$, there exist codimension $2$ submanifolds $\cM \subseteq \R^{2m+2}$ satisfying (CC) for which the $Q^{2m-1}$ term cannot be replaced with $Q^r$ for some $r < 2m-1$.




\subsection{A variant: counting Gaussian rationals}\label{subsec: Gaussian} In the complex setting, there is another natural counting problem for rational points near a manifold. To describe this, given $z \in \C$, define $|z|_{\infty} := \max\{|\fkR z|,|\fkI z|\}$. Let $m$, $d \in \N$ with $m < d$ and $\cM\subset \C^d$ be a complex $m$-dimensional manifold and consider
\begin{equation*}
  N_{\mathcal{M}}^\C(Q,\delta)  = \# \big\{ \bdb/q : \mathrm{dist}(\bdb/q, \cM) \le \delta/|q|_{\infty},\, \bdb \in \Z[i]^d \textrm{ and } q\in \Z[i], \, 0 < |q|_{\infty} \le Q \bigr\},  
\end{equation*}
where now the denominator $q$ varies over \textit{Gaussian integers} satisfying $0 < |q|_{\infty}  \leq Q$ and 
\begin{equation*}
\mathrm{dist} (\bdz, \cM) = \inf_{\bdw = (w_1, \dots, w_d) \in \cM} \max_{1\le i \le d}|z_i - w_i|_{\infty}   \qquad \textrm{for all $\bdz = (z_1, \dots, z_d) \in \C^d$.}
\end{equation*}
Of course $\cM$ can be thought of as a real $2m$-dimensional submanifold of $\R^{2d}$ and all we are doing here is considering a larger class of rational points in $\Q^{2d}$ whose denominator has $\ell^{\infty}$-norm at most $Q$. Considering only real denominators $q = q_1 + i 0$ (so that $|q|_{\infty} = q_1$) leads back to the original problem and hence 
\begin{equation*}
 N_\cM(Q,\delta) \le N_\cM^\C(Q,\delta).   
\end{equation*}
The choice of $\ell^{\infty}$-norm on $\C$ (rather than the usual complex modulus) helps facilitate the passage between complex and real manifolds. 

By adapting the methods of Huang~\cite{Huang2020}, and incorporating modifications of \cite{Srivastava2025} into our argument, we are able to prove sharp bounds for $N_{\mathcal{M}}^\C(Q,\delta)$ in the hypersurface case. To give a precise statement of our result, fix a dimensional constant $\kappa_m \geq 1$, chosen sufficiently large for our forthcoming purposes, and let 
\begin{equation*}
    \cE(Q) := 
    \begin{cases}
        \log^{\kappa_m}(Q) & \textrm{if $m \geq 3$,} \\
        e^{\kappa_2\sqrt{\log Q \log \log Q}} & \textrm{if $m = 2$.}
    \end{cases}
\end{equation*}
With this definition, our bounds have the following form. 

\begin{theorem}\label{thm: main cplx} Let $m \in \N$ and $\cM \subseteq \C^{m+1}$ be a complex $m$-dimensional hypersurface satisfying \emph{($\C$-CC)}. 
\begin{enumerate}[a)]
    \item If $m \geq 2$, then $N^\C_{\cM}(Q,\delta) \ll \delta^2 Q^{2m+2} + Q^{2m}\cE(Q)$;\\
    \item If $m = 1$, then $N^\C_{\cM}(Q,\delta) \ll \delta^2 Q^4 + Q^3 \log^4 Q$. 
\end{enumerate}
\end{theorem}
As we shall see below, for $m \geq 4$, Theorem~\ref{thm: main cplx} establishes the sharp bound, up to the subpolynomial losses represented by the $\cE(Q)$ factor. For $m = 1$ it may be possible to improve the $Q^3$ power in the second term to $Q^2$. This situation is analogous to the situation for real $n$-dimensional hypersurfaces in $\R^{n+1}$. Indeed, for real hypersurfaces, Huang's method \cite{Huang2020} is very effective in higher dimensions, but for $n=1$ fails to recover the sharp bounds of Huxley~\cite{Huxley1994} and Vaughan--Velani~\cite{VV2006} for planar curves.

To see sharpness for the $m \geq 4$ case, we revisit Example~\ref{ex: sphere} and adapt it to the Gaussian rational setting. 

\begin{example}[Complex sphere]\label{ex: sphere Gauss} The analysis of Example~\ref{ex: sphere} can also be used to estimate the Gaussian rational count $N_{\bbS^m_{\C}}^{\C}(Q,0)$. In particular, for $m \geq 4$ we can again use the asymptotic \eqref{eq: Siegel asympt} to deduce that
\begin{equation*}
    N_{\bbS_{\C}^m}^{\C}(Q, 0) = \sum_{\substack{q \in \Z[i] \\ 0 < |q|_{\infty} \leq Q}} A_{m+1}(q^2) \gtrsim Q^{2m},
\end{equation*}
provided $Q \geq 1$ is sufficiently large, depending on $m$.
\end{example}

Although Theorem~\ref{thm: main cplx} may be of interest in its own right, our primary motivation is that it provides evidence for Conjecture~\ref{conj: codim 2} above. Indeed, it is notable that the sharp examples for Theorem~\ref{thm: main cplx} are essentially the same as the posited sharp examples for Conjecture~\ref{conj: codim 2}.




\subsection{Low dimensional cases}\label{subsec: n=2} 

Conjecture~\ref{conj: codim 2} asserts that $\fke(n, 2) = 2$ holds for all \textit{sufficiently large} even $n \in \N$. The restriction to large $n$ is necessary: $\fke(2, 2) = 2$ is easily seen to be false. 

\begin{example}[Complex parabola]
Consider the portion of the complex parabola
\begin{equation*}
  \bbP^1_{\C} := \{(z, z^2)  : z \in \C,\, -1 < \fkR z,\, \fkI z < 1 \}.
\end{equation*} 
It is straightforward to verify $\bbP^1_{\C}$ satisfies ($\C$-CC). 
 We identify $\bbP_{\C}^1$ with the codimension $2$ real submanifold of $\R^4$ given by
\begin{equation*}
     \bbP^1_{\C} = \{(x, y, x^2 - y^2, 2xy)  :  -1 < x,\, y < 1 \}.
\end{equation*}
By either Lemma~\ref{lem: ND vs CND} or direct verification, this satisfies (CC).

Let $Q \in \N$ be large. If $a$, $b$, $u \in \Z$ satisfy $1 \leq a < u \leq \sqrt{Q}$ and $1 \leq b < u \leq \sqrt{Q}$, then
\begin{equation*}
 (x,y) := \Big(\frac{a}{u}, \frac{b}{u}\Big) \quad \textrm{satisfies} \quad (x, y, x^2 - y^2, 2xy) \in \bbP^1_{\C} \quad \textrm{and} \quad (x, y, x^2 - y^2, 2xy) = \frac{\bdp}{q} 
\end{equation*}
where $1 \leq q \leq Q$ and $\bdp \in \Z^4$. Thus, 
\begin{equation*}
    N_{\bbP^1_{\C}}(Q,0) \gg \sum_{u = 1}^{\floor{\sqrt{Q}}} u^2 \gg Q^{3/2}. 
\end{equation*}
If $\fke(n, 2) = 2$ were to hold, then this would imply $N_{\bbP^1_{\C}}(Q,0) \ll Q$, which contradicts the above bound.
\end{example}

We can also consider Gaussian rationals lying inside the complex parabola. 

\begin{example}[Complex parabola] Let $Q \in \N$ be large. If $a$, $b$, $u$, $v \in \Z$ satisfy $1 \leq a < u \leq \sqrt{Q/2}$ and $1 \leq b < v \leq \sqrt{Q/2}$, then
\begin{equation*}
 z := \frac{a+ib}{u + iv} \quad \textrm{satisfies} \quad (z, z^2) \in \bbP^1_{\C} \quad \textrm{and} \quad (z, z^2) = \frac{\bdb}{q} 
\end{equation*}
where $q \in \Z[i]$ satisfies $0 < |q|_{\infty} \leq Q$ and $\bdb \in \Z[i]^2$. Thus, 
\begin{equation*}
    N_{\bbP^1_{\C}}^\C(Q,0) \geq \sum_{v = 1}^{\floor{\sqrt{Q/2}}} \sum_{u = 1}^{\floor{\sqrt{Q/2}}} uv \gg Q^2. 
\end{equation*}
This example suggests that it may be possible to improve the $Q^3$ power to $Q^2$ in the $m = 1$ case of Theorem~\ref{thm: main cplx}. 
\end{example}

\subsection*{Structure of the paper} In \S\ref{sec: prelim} we introduce some notation and standard technical tools used throughout the paper. In \S\ref{sec: int dim} we prove the refined estimate in Theorem~\ref{thm: refined count}, relying heavily on earlier results and methods from \cite{Srivastava2025}. In \S\ref{sec: examples}, we present the proof of Lemma~\ref{lem: ND vs CND} and provide further details of the complex sphere example discussed in \S\ref{subsec: codim 2 conj} and \S\ref{subsec: Gaussian}. Finally, in \S\ref{sec: Guassian count} we prove Theorem~\ref{thm: main cplx} using Huang's Fourier analytic method \cite{Huang2020} and modifications thereof from \cite{Srivastava2025}. Appended is a proof of a simple summation-by-parts lemma which features in our arguments.

\subsection*{Acknowledgements} The first author is supported by New Investigator Award UKRI097. He wishes to thank Damaris Schindler for an interesting talk on related topics at the Hausdorff Institute in Bonn, April 2025, which partly inspired this work. The second author was partially supported by the Deutsche Forschungsgemeinschaft (DFG, German Research Foundation) under Germany's Excellence Strategy-EXC-2047/1-390685813 and an Argelander grant from the University of Bonn.




\section{Notation and technical tools}
\label{sec: prelim}
We summarise notation and technical tools used throughout the paper.




\subsection{Notational conventions} 

\begin{itemize}
    \item Given $d \in \N$ and $\bdx = (x_1, \dots, x_d) \in \R^d$, we let $|\bdx|_{\infty} := \max \{|x_1|, \dots, |x_d|\}$. 
    \item We let $e \colon \R \to \C$ denote the Fourier character $e(x) := e^{2 \pi i x}$ for $x \in \R$. 
    \item We denote by $\|x\| := \min\{|x - z| : z \in \Z\}$ the distance of $x \in \R$ to the nearest integer.
    \item Given a set $E \subseteq \R$, we let $\mathbbm{1}_E \colon \R \to \{0,1\}$ denote its characteristic function.
\end{itemize}




\subsection{The Fej\'er kernel} For $D\in \N$, let $\cF_D \colon \R \to [0,\infty)$ denote the Fej\'er kernel of degree $D$, given by
\begin{equation*}
    \cF_D(\theta)= \sum_{d=-D}^{D}\frac{D-|d|}{D^2} e(d\theta)=\left(\frac{\sin(\pi D\theta)}{D\sin(\pi\theta)}\right)^2.
\end{equation*}
Suppose $D := \left\lfloor \frac{1}{2\delta}\right\rfloor$ for some $\delta \in (0, 1/2)$. Since $|\sin (\pi x)|\geq 2x$ for $x \in [0, 1/2]$, we have
\begin{equation*}
    \left(\frac{\sin(\pi D\theta)}{D\sin(\pi\theta)}\right)^2\geq \left(\frac{2 D\|\theta\|}{D\pi \|\theta\|}\right)^2=\frac{4}{\pi^2}
\end{equation*}
whenever $\|\theta\|\in [0, \delta]$.
In other words, 
\begin{equation}
    \label{eq fejer char est}
    \mathbbm{1}_{\{x\in \R: \|x\|\leq \delta\}}(\theta)\leq \frac{\pi^2}{4} \cF_D(\theta) \qquad \textrm{for all $\theta \in \R$.}
\end{equation}




\subsection{Stationary phase}

For $n$, $d \in \N$, let $Y \subseteq \R^n$ and $Z \subseteq \R^d$ be open sets of diameter at most $1$ and $\phi \colon Y \times Z \to \R$ be a smooth function. Suppose there exists a smooth function $\bdy \colon Z \to Y$ such that for each $\bdz \in Z$, the point $\bdy = \bdy(\bdz)$ is the unique solution to 
\begin{equation*}
    \nabla_{\bdy} \phi(\bdy; \bdz) = 0 \qquad \textrm{for $\bdy \in Y$.}
\end{equation*}
Further suppose there exists some constant $c > 0$ such that
\begin{equation}\label{eq: stationary phase 1}
    |\det (\partial_{\bdy \bdy}^2 \phi)(\bdy(\bdz), \bdz)| \geq c > 0 \qquad \textrm{for all $\bdz \in Z$.}
\end{equation}

Let $w \in C^{\infty}((0, \infty)\times\R^n \times \R^d)$ with $\supp w \subseteq (0, \infty) \times Y \times Z$ and suppose there exists a constant $C > 0$ such that
\begin{equation}\label{eq: stationary phase 2}
 \sup_{(\lambda;  \bdy; \bdz) \in (0, \infty) \times \R^n \times \R^d}|\partial^N_{\lambda}  \partial^{\bdalpha}_{\bdy}\partial^{\bdbeta}_{\bdz} w(\lambda; \bdy; \bdz)| \leq C (1 + \lambda)^{-N}
\end{equation}
for all $N \in \N_0$ and multi-indices $\bdalpha = (\alpha_1, \dots, \alpha_n)\in \N_0^n$, $\bdbeta = (\beta_1, \dots, \beta_d) \in \N_0^d$ satisfying $0 \leq N$, $\alpha_1 + \cdots + \alpha_n$, $\beta_1 + \cdots + \beta_d \leq 100n$. 

Define the oscillatory integral
\begin{equation*}
    I(\lambda; \bdz) := \int_{\R^n} e(\lambda \phi(\bdy; \bdz)) w(\lambda; \bdy; \bdz)\,\ud \bdy \qquad \textrm{for $(\lambda; \bdz) \in (0,\infty) \times Z$.} 
\end{equation*}
The following (standard) lemma provides a useful asymptotic for $I(\lambda; \bdz)$.

\begin{lemma}
[Variable coefficient stationary phase]
\label{lem: stationary phase} With the above setup,
\begin{equation*}
    I(\lambda; \bdz) = \lambda^{-n/2} e\big(\lambda \phi(\bdy(\bdz); \bdz)\big)a(\lambda;\bdz) \qquad \textrm{for all $(\lambda; \bdz) \in (0, \infty) \times Z$}
\end{equation*}
where $a \in C^{\infty}((0, \infty) \times \R^d)$ satisfies
\begin{equation*}
\sup_{(\lambda; \bdz) \in (0, \infty) \times \R^d} |\partial_{\lambda}^{\iota} \partial_{\bdz}^{\bdbeta} a(\lambda; \bdz)| \ll_{c, C} \lambda^{-\iota} \qquad \textrm{for $\iota \in \{0,1\}$ and $\bdbeta \in \{0,1\}^d$.} 
\end{equation*}
The implicit constant depends only on $n$, $d$ and $c$ from \eqref{eq: stationary phase 1} and $C$ from \eqref{eq: stationary phase 2}. 
\end{lemma}
\begin{proof}
    This follows, for example, from \cite[Corollary 1.1.8]{Sogge_book}, with some minor modifications. The dependence of the implied constant is not explicitly discussed in the reference, but can easily be deduced by examining the proof. The requirement that we bound derivatives up to order $100n$ in \eqref{eq: stationary phase 1} is to ensure this dependence is somewhat of an overkill, but nevertheless suffices for our purposes.
\end{proof}

\subsection{Summation-by-parts}\label{subsec: sum by parts} Given $\lambda \geq 1$, let $\fkJ^n(\lambda)$ denote the collection of all axis-parallel rectangles $\bdJ := J_1 \times \cdots \times J_n$ where $J_1, \dots, J_n \subseteq [0,\lambda]$ are closed intervals with integer endpoints. For $\bdJ \in \fkJ^n(\lambda)$ and $B > 0$, we define the symbol class $\fkA(\lambda; \bdJ; B)$ as the collection of all functions $\fka \in C^{\infty}(\bdJ)$ which satisfy
\begin{equation}
    \label{eq sbp wt cond}
    \sum_{\alpha \in \{0,1\}^n} \lambda^{|\alpha|}\sup_{\bdu \in\bdJ}|\partial_{\bdu}^{\alpha}\fka(\bdu)| \leq B.
\end{equation}
For $\fka \in \fkA(\lambda; \bdJ; B)$ and $\bdx \in \R^n$, we define the exponential sum 
\begin{equation*}
    S_{\fka}(\bdx, \bdJ) := \sum_{\bdj \in \bdJ \cap \N_0^n} \fka(\bdj) e(\bdj \cdot \bdx).
\end{equation*}
With this setup, we record the following useful summation-by-parts lemma.

\begin{lemma}\label{lem: summation by parts} Let $\fka \in \fkA(\lambda; \bdJ; B)$ for some $\lambda \geq 1$, $\bdJ \in \fkJ^n(\lambda)$ and $B > 0$. Then
\begin{equation*}
    |S_{\fka}(\bdx, \bdJ)| \leq B \prod_{s = 1}^n \min\Big\{\lambda, \frac{1}{\|x_s\|}\Big\} \qquad \textrm{for all $\bdx = (x_1, \dots, x_n) \in \bdJ$.}
\end{equation*} 
\end{lemma}

The proof of Lemma~\ref{lem: summation by parts} is elementary. We include the details in the appendix for completeness.




\section{Manifolds of intermediate dimension: Proof of Theorem \ref{thm: aniso refined}}\label{sec: int dim}




\subsection{A non-isotropic generalisation} As in \cite{Srivastava2025}, we shall in fact establish a more general version of Theorem \ref{thm: refined count} for the number of rational points with bounded denominators contained in a non-isotropic neighborhood of $\cM$. It is not hard to see that under the parametrization \eqref{eq: real graph}, the isotropic counting function $N_{\cM}(Q, \delta)$ can be dominated by a constant times
\begin{equation*}
    \#\big\{ (\bda, q) \in \Z^{n+1} : 1 \leq q \leq Q, \, \bda/q\in U,\, \|qf_r(\bda/q)\| \leq \delta/q,\, 1\leq r\leq R \big\}.
\end{equation*}
For $Q\in \N$ and $\bddelta=(\delta_1, \ldots, \delta_R)\in (0, 1/4)^R$, we define the non-isotropic counting function
\begin{equation*}
  \cN_{\cM}(Q, \bddelta):=\#\left\{(\ba, q)\in \mathbb{Z}^{n+1}: 1\leq q\leq Q,\, \bda/q\in U,\, \left\|qf_r(\ba/q)\right\|\leq \delta_r/q \text{ for }1\leq r\leq R\right\}.  
\end{equation*}
When $\delta_1=\cdots=\delta_R=\delta$, we shall simply refer the above as $\cN_{\cM}(Q, \delta)$.

Let
\begin{equation*}
    \delp:= \prod_{r=1}^R\delta_r \qquad \textrm{and} \qquad \bddelta^{\wedge} := \min_{1\leq r \leq R} \delta_r.
\end{equation*}
Theorem~\ref{thm: refined count} is then a consequence of the following estimate for rational points in non-isotropic neighbourhoods of $\cM$, which refines \cite[Theorem 1.9]{Srivastava2025}.

\begin{theorem}
    \label{thm: aniso refined}
    For an admissible pair $(n,R)$, let $\cM$ be a codimension $R$ submanifold  of $\R^{n+R}$ satisfying (CC). For all $\nu>0$, the bound
\begin{equation}\label{eq: aniso refined}
    N_{\cM}(Q, \bddelta) \ll_\nu \delp Q^{n+1} + \bddelta^{\times}(\bddelta^{\wedge})^{-R} Q^{n + 1 - \frac{(n+2)R}{n+2R} + \nu}
\end{equation}
holds for all $Q \in \N$ and $\bddelta \in (0,1/4)^R$.
\end{theorem}

Taking $\delta_1 = \cdots = \delta_R = \delta$, it follows that $\bddelta^{\times} = \delta^R$ and $\bddelta^{\times} (\bddelta^{\wedge})^{-R} = 1$. Thus, Theorem~\ref{thm: aniso refined} implies that,  for all $\nu > 0$, we have
\begin{equation*}
    N_{\cM}(Q, \delta) \ll_{\nu} \delta^R Q^{n+1} + Q^{n + 1 - \frac{(n+2)R}{n+2R} + \nu}.
\end{equation*}
Recalling the definition of the exponent $\fke(n,R)$, this directly implies Theorem~\ref{thm: refined count}.

\begin{remark}
    The shape of the estimate in Theorem~\ref{thm: aniso refined} is slightly different from that of its counterpart \cite[Theorem 1.9]{Srivastava2025}. For instance, the right-hand side of \eqref{eq: aniso refined} involves only two terms, whilst the right-hand side of the corresponding estimate in \cite[Theorem 1.9]{Srivastava2025} has three. The alternative form of the estimate in Theorem~\ref{thm: aniso refined} allows us to streamline certain aspects of the argument. 
\end{remark}




\subsection{Preliminaries}\label{subsec: prelim} Suppose $\cM \subset \R^{n+R}$ is a smooth surface satisfying the curvature condition (CC). By working locally, we may assume 
\begin{equation*}
    \cM = \{(\bdx, f(\bdx)) : \bdx \in U \}
\end{equation*}
where $U \subset \R^n$ is a bounded, open set and $f = (f_1, \dots, f_R) \colon U \to \R^R$ is a smooth function satisfying \eqref{eq: CC}. Given $\bdtheta \in \R^R$, define
\begin{equation*}
    f_{\bdtheta} \colon \R^n \to \R, \qquad f_{\bdtheta}(\bdx) := f(\bdx) \cdot \bdtheta = \sum_{r=1}^R \theta_r f_r(\bdx). 
\end{equation*}
By a quantitative version of the inverse function theorem (see, for example, \cite[\S8]{Christ1985}), we may cover $U$ by small balls $B(\bdx_0, \tau^2)$ with $0 < \tau < 1$ such that for $\bdtheta \in \R^R$ satisfying $1 \leq |\bdtheta|_{\infty} \leq 2$ the following hold:
\begin{itemize}
    \item There exist constants $C_1$, $C_2 \geq 1$ depending only on $f$ such that
  \begin{equation}\label{eq: codim R ift 1}
  \begin{split}
      (\nabla_{\bdx} f_{\bdtheta})\bigl(B(\bdx_0, \tau^2)\bigr) &\subseteq B((\nabla_{\bdx} f_{\bdtheta})(\bdx_0), C_1 \tau^2) \\
      & \subseteq B((\nabla_{\bdx} f_{\bdtheta})(\bdx_0), 2 C_1 \tau^2) \subseteq 
(\nabla_{\bdx} f_{\bdtheta})\bigl(B(\bdx_0, C_2 \tau)\bigr);
  \end{split}
\end{equation}
\item $\nabla_{\bdx} f_{\bdtheta}$ is a diffeomorphism between $B(\bdx_0, C_2 \tau)$ and
$(\nabla_{\bdx} f_{\bdtheta})\bigl(B(\bdx_0, C_2 \tau)\bigr)$.
\end{itemize}




\subsubsection*{Smooth counting function} Henceforth, we fix one of the balls $B(\bdx_0, \tau^2)$ as above and let $\varepsilon := \tau^2/2$. Using a smooth partition of unity, we can reduce matters to studying the following smooth variant of $\cN_{\cM}(Q, \bddelta)$.  We let $\cW_{\cM}$ denote the set of all smooth weights $w \colon \R^n\to [0, 1]$ satisfying $\supp w \subseteq B(\bdx_0, 2\varepsilon)$. For $w \in \cW_{\cM}$, given $Q \in \N$ and $\bddelta = (\delta_1, \dots, \delta_R) \in (0,1/4)^R$, we define 
\begin{equation}\label{eq: codim R smooth count}
    \fkN_{\cM, w}(Q, \bddelta) := \sum_{q = 1}^Q \sum_{\substack{\bda\in \Z^n \\ \|qf_r(\bda/q)\|\le \delta_r \\ 1 \leq r \leq R}} w(\bda/q).
\end{equation}
If $\bddelta = (\delta, \dots, \delta)$ for some $\delta \in (0, 1/4)$, then we simply write $\fkN_{\cM, w}(Q, \delta)$ for $\fkN_{\cM, w}(Q, \bddelta)$. With this notation, to prove Theorem~\ref{thm: aniso refined}, our goal is to show that for all $\nu > 0$ the inequality
\begin{equation*}
\fkN_{\cM, w}(Q, \delta)  \ll_\nu \delp Q^{n+1} + \bddelta^{\times}(\bddelta^{\wedge})^{-R} Q^{n + 1 - \frac{(n+2)R}{n+2R} + \nu}
\end{equation*}
holds, at least for a suitable choice of $w \in \cW_{\cM}$.




\subsubsection*{Duality} Given $\bdtheta \in \R^n$ satisfying $1 \leq |\bdtheta|_{\infty} \leq 2$, recall that $\nabla_{\bdx} f_{\bdtheta} \colon \cD \to \cD_{\bdtheta}^*$ is a diffeomorphism between 
\begin{equation*}
   \cD := B(\bdx_0, C_2 \tau) \qquad \textrm{and} \qquad \cD_{\bdtheta}^* := (\nabla_{\bdz} f_{\bdtheta})(B(\bdx_0, C_2 \tau)).
\end{equation*}
The Legendre transform $f_{\bdtheta}^{*} \colon \cD_{\bdtheta}^* \to \R$ of $f_{\bdtheta}$ is then defined by
\begin{equation*}
    f_{\bdtheta}^{*}(\bdy) := - f_{\bdtheta}(\bdx) + \bdx \cdot \bdy \qquad \textrm{where $\bdx \in \cD$ satisfies $\bdy = \nabla_{\bdx} f_{\bdtheta}(\bdx)$}
\end{equation*}
for $\bdy \in \cD_{\bdtheta}^*$. It is easy to check that
\begin{gather}\label{eq: L trans 1}
   \nabla_{\bdy} f_{\bdtheta}^{*} (\bdy) = (\nabla_{\bdx} f_{\bdtheta})^{-1}(\bdy); \qquad f_{\bdtheta}^{**}(\bdx) = f_{\bdtheta}(\bdx); \\
   \label{eq: L trans 2}
   (f_{\bdtheta}^{*})''(\bdy) = (f_{\bdtheta}'')^{-1} ((\nabla_{\bdx} f_{\bdtheta})^{-1}(\bdy)). 
\end{gather}
Hence, by our hypothesis \eqref{eq: CC} on $f$, we see that $f_{\bdtheta}^{*}$ also satisfies \eqref{eq: CC}, in the sense that $\det (f_{\bdtheta}^{*})''(\bdy) \neq 0$ for all $\bdy \in \cD_{\bdtheta}^*$. We set
\begin{equation*}
    w_{\bdtheta}^{*}(\bdy) := \frac{w \circ (\nabla_{\bdx} f_{\bdtheta})^{-1}(\bdy)}{|\det f''_{\bdtheta} \circ (\nabla_{\bdx} f_{\bdtheta})^{-1}(\bdy)|}.
\end{equation*}
Given $Q_*\in \N_0$ and $\delta_* \in (0,1/4)$, we then define the \textit{dual counting function}
\begin{equation}\label{eq: dual count}
\fkN_{\cM, w}^*(Q_*,\delta_*) := \sum_{\substack{\bdj \in \N_0^R \\
0 < |\bdj|_{\infty} \leq Q_*}} \;\; \sum_{\substack{\bda \in \Z^n \\ \||\bdj|_{\infty} f_{\bdj/|\bdj|_{\infty}}^{*}(\bda/|\bdj|_{\infty})\| \le \delta_{*}}} w_{\bdj/|\bdj|_{\infty}}^{*}(\bda/|\bdj|_{\infty}).
\end{equation}
Note that $\fkN_{\cM, w}^*(Q_*,\delta_*)$ has a rather different form compared to the original counting function $\fkN_{\cM, w}(Q,\bddelta)$.




\subsubsection*{Weights} For technical reasons, it is necessary to work with a sequence of weights with nested supports, rather than a single weight $w \in \cW_{\cM}$. This forces us to quantify the sizes of various derivatives of the weight. 

\begin{definition}\label{dfn: weight ext} For $\rho \geq 1$, we define $\cW_{\cM}(\rho)$ to be the collection of all $w \in \cW_{\cM}$ satisfying
\begin{equation*}
        \|\partial^{\bdalpha}w\|_{L^{\infty}(\R^n)} \leq \rho \quad \textrm{for all $\bdalpha\in \N_0^n$ with $0 \leq \alpha_1 + \cdots + \alpha_n \leq 100n$.} 
    \end{equation*}
\end{definition}

We shall frequently pass from one weight $w \in \cW_{\cM}$ to a weight $\tilde{w} \in \cW_{\cM}$ with slightly larger support. 

\begin{definition}
Given $w \in \cW_{\cM}$ and $0 < h \leq 1$, we let $\cW_{\cM}(w; h)$ denote the set of all weights $\widetilde{w} \in \cW_{\cM}$ for which there exist $0< \varepsilon_1 < \varepsilon_2 \leq 2\varepsilon$ with $\varepsilon_2 - \varepsilon_1 \geq h$  such that
\begin{equation*}
     \supp  w \subseteq B(\bdx_0, \varepsilon_1) \qquad \textrm{and} \qquad \widetilde{w}(\bdy) = 1 \quad \textrm{for all $\bdy \in B(\bdx_0, \varepsilon_2)$.} 
\end{equation*}
\end{definition}




\subsection{Self-improving mechanism}

Theorem~\ref{thm: aniso refined} is a direct consequence of the following two lemmas, which together form a bootstrapping or \textit{self-improving mechanism} for estimating our counting functions.

\begin{lemma}\label{lem: codim R self imp} Let $0 < h \leq 1$ and $\rho \geq 1$ and suppose $w \in \cW_{\cM}(\rho)$ and $\widetilde{w} \in \cW_{\cM}(w; h)$. Further suppose that for all $\bddelta \in (0, 1/4)^R$ the bound
\begin{equation}\label{eq: codim R hypothesis}
\fkN_{\cM, \widetilde{w}}(Q, \bddelta) \leq A \big(\bddelta^{\times}Q^{n+1} + \bddelta^{\times} (\bddelta^{\wedge})^{-R} Q^{\beta} \log^{\gamma} Q\big). 
\end{equation}
holds for some $\frac{n(n+R+1)}{n+2R} < \beta \le n+1$, $\gamma \geq 0$ and $A \geq 1$. Then, for all $\delta_* \in (0, 1/4)$, we have
\begin{equation}\label{eq: codim R improv}
\fkN_{\cM, w}^*(Q_*,\delta_*) \ll_{h, \rho, \beta}  A\big( \delta_* Q_*^{n+R} +  Q_*^{n + R - \frac{n}{2\beta - n}} \log^{\frac{2(R-1 + \gamma)}{2\beta - n}} Q_*\big).
\end{equation}
\end{lemma}

\begin{lemma}[{\textit{c.f.} \cite[Proposition 3.3]{Srivastava2025}}]\label{lem: codim R self imp *} Let $\rho \geq 1$ and $w \in \cW_{\cM}(\rho)$. Suppose that for all $\delta_* \in (0, 1/4)$, the bound
\begin{equation*}
\fkN_{\cM, w}^*(Q_*,\delta_*) \le  A\big( \delta_* Q_*^{n+R} +  Q_*^{\alpha} \log^{\gamma} Q_*\big)
\end{equation*} 
holds for some $\frac{n(n + R +1)}{n+2} < \alpha \le n+R$, $\gamma \geq 0$ and $A \geq 1$. Then, for all $\bddelta \in (0, 1/4)^R$, we have
\begin{equation*}
\fkN_{\cM, w}(Q, \bddelta) \ll_{\rho, \alpha} A\big( \bddelta^{\times}Q^{n+1} +  \bddelta^{\times} (\bddelta^{\wedge})^{-R} Q^{n + 1 - \frac{nR}{2\alpha - n}} \log^{\frac{2R\gamma}{2\alpha - n}} Q\big). 
\end{equation*}
\end{lemma}

\begin{remark}\label{rmk: alpha beta} We observe the following relationship between the conditions on the exponents $\beta$ and $\alpha$ in Lemma~\ref{lem: codim R self imp} and Lemma~\ref{lem: codim R self imp *}.
\begin{itemize}
    \item If $\beta > \frac{n(n+R+1)}{n+2R}$, then $\frac{n}{2\beta - n} < \frac{n+2R}{n+2}$. Defining $\alpha := n + R - \frac{n}{2\beta - n}$, this in turn implies that $\alpha > \frac{n(n+R+1)}{n+2}$. 
    \item If $\alpha > \frac{n(n+R+1)}{n+2}$, then $\frac{n}{2\alpha - n} < \frac{n+2}{n+2R}$. Defining $\beta := n + 1 - \frac{nR}{2\alpha - n}$, this in turn implies that $\beta > \frac{n(n+R+1)}{n+2R}$. 
\end{itemize}
\end{remark}

Lemma~\ref{lem: codim R self imp} and Lemma~\ref{lem: codim R self imp *} are counterparts of \cite[Proposition 3.2]{Srivastava2025} and \cite[Proposition 3.3]{Srivastava2025}, respectively. More precisely, Lemma~\ref{lem: codim R self imp} refines \cite[Proposition 3.2]{Srivastava2025} by extending the range of the exponent $\beta$, and is the source of the improvement in Theorem~\ref{thm: refined count} over the results in \cite{Srivastava2025}. On the other hand, Lemma~\ref{lem: codim R self imp *} is a minor variant of \cite[Proposition 3.3]{Srivastava2025} and both Lemma~\ref{lem: codim R self imp *} and \cite[Proposition 3.3]{Srivastava2025} can be thought of as the same strength. 

\begin{remark} In both Lemma~\ref{lem: codim R self imp} and Lemma~\ref{lem: codim R self imp *}, the shape of the estimates is slightly different compared to those appearing in \cite{Srivastava2025}, with only two terms (rather than three) appearing on the right-hand sides of the inequalities. The modified form of the estimates helps to streamline the argument. Indeed, the middle term on the right-hand sides in \cite{Srivastava2025} can be dominated by the geometric mean of the terms appearing in the inequalities above.
\end{remark}

We shall only present the proof of Lemma~\ref{lem: codim R self imp}. The argument used to prove \cite[Proposition 3.3]{Srivastava2025} can easily be adapted to prove Lemma~\ref{lem: codim R self imp *}. In what follows, it is important to note that, in order to establish \eqref{eq: codim R improv}, we need the bound \eqref{eq: codim R hypothesis} for all $\bddelta = (\delta_1, \dots, \delta_R) \in (0,1/4)^R$ and not just the case $\delta_1= \cdots =  \delta_R$.

\begin{proof}[Proof (of Lemma~\ref{lem: codim R self imp})] The argument is rather lengthy, so we break it into steps. \medskip

\noindent \textit{Step 0: Dyadic decomposition}. Given $Q_*\in \N_0$ and $\delta_* \in (0,1/4)$, recalling the definition \eqref{eq: dual count}, we may bound 
\begin{equation*}
    \fkN_{\cM, w}^*(Q_*,\delta_*)  \leq \sum_{s=1}^R \sum_{\substack{\bdj \in \N_0^R \\
0 < |\bdj|_{\infty} = j_s \leq Q_*}} \;\; \sum_{\substack{\bda \in \Z^n \\ \|j_s f_{\bdj/j_s}^{*}(\bda/j_s)\| \le \delta_{*}}} w_{\bdj/j_s}^{*}(\bda/j_s).
\end{equation*}
In view of what follows, it is convenient to relax the summation in $\bdj$. In particular,
\begin{equation*}
\fkN_{\cM, w}^*(Q_*,\delta_*)  \leq \sum_{s=1}^R \fkN_{\cM, w}^{*,s}(Q_*,\delta_*) 
\end{equation*}
where
\begin{equation}\label{eq: relax}
  \fkN_{\cM, w}^{*,s}(Q_*,\delta_*) :=  \sum_{\substack{\bdj \in \N_0^R \\ 0 < |\bdj|_{\infty} \leq 2j_s \leq 2Q_*}} \;\;\sum_{\substack{\bda \in \Z^n \\ \|j_s f_{\bdj/j_s}^{*}(\bda/j_s)\| \le \delta_{*}}} w_{\bdj/j_s}^{*}(\bda/j_s).
\end{equation} 

We further decompose the $\fkN^{*, s}_{\cM, w}(Q_*, \delta_*)$ dyadically. For $\ell \in \N$, let
\begin{equation*}
\cJ_{\ell}^s := \Big\{\bdj\in \N_0^R : 2^{\ell-1} < j_s \leq 2^{\ell} \textrm{ and } |\bdj|_{\infty} \leq 2^{\ell+1}\Big\}.    
\end{equation*}
With this definition, the range of $\bdj$ summation in \eqref{eq: relax} satisfies
\begin{equation*}
   \big\{\bdj \in \N_0^R : 0 < |\bdj|_{\infty} \leq 2j_s \leq 2Q_*\big\} \subseteq \bigcup_{\ell = 1}^{\ceil{\log_2Q_*}} \cJ_{\ell}^s
\end{equation*}
and so we may bound 
\begin{equation*}
    \fkN^{*, s}_{\cM, w}(Q_*, \delta_*) \leq \sum_{\ell=1}^{\ceil{\log_2 Q_*}} \fkN^{*, s, \ell}_{\cM, w} (Q_*, \delta_*)
\end{equation*}
where
\begin{equation*}
    \fkN^{*, s, \ell}_{\cM, w} (Q_*, \delta_*) := \sum_{\bdj\in \cJ_{\ell}^s}\sum_{\substack{\bda \in \Z^n \\ \|j_s f^*_{\bdj/j_s}(\ba/j_s)\| \leq \delta_*}}\, w_{\bdj/j_s}^*(\bda/j_s).
\end{equation*}
Note that for $\bdj \in \cJ_{\ell}^s$ we have $1 \leq |\bdj/j_s|_{\infty} \leq 2$ and so $\nabla_{\bdx} f_{\bdj/j_s} \colon \cD \to \cD_{\bdj/j_s}^*$ is a diffeomorphism. Thus, each dyadic counting function $\fkN^{*, s, \ell}_{\cM, w} (Q_*, \delta_*)$ is well defined.\medskip

\noindent \textit{Step 1: Fourier reformulation}. Let $D := \floor{\frac{1}{2\delta_*}}$, noting $D \geq \frac{1}{2\delta_*} - 1 \geq \frac{1}{4\delta_*}$. For each $1 \leq \ell \leq \ceil{\log_2 Q_*}$, we introduce the degree $D$ Fej\'er kernel as in \eqref{eq fejer char est} to bound
\begin{equation*}
\fkN^{*, s, \ell}_{\cM, w} (Q_*, \delta_*) \leq \delta_* \sum_{d=-D}^D a_{d,\delta_*} \sum_{\bdj \in \cJ_{\ell}^s}\sum_{\substack{\bda \in \Z^n}}w_{\bdj/j_s}^*(\bda/j_s) e\big(dj_s f^*_{\bdj/j_s}(\ba/j_s)\big),
\end{equation*}
where the real coefficients $a_{d, \delta_*}$ satisfy $a_{-d, \delta_*} = a_{d, \delta_*}$ and $0 \leq a_{d, \delta_*} \leq \pi^2$ for all $d \in \Z$. The right-hand side of the above inequality can be split into two terms
\begin{equation*}
    \pi^2\delta_* \sum_{\bdj \in \cJ_{\ell}^s}\sum_{\substack{\bda \in \Z^n}}w_{\bdj/j_s}^*(\bda/j_s) + \delta_* 2\fkR \sum_{d=1}^D a_{d,\delta_*} \sum_{\bdj \in \cJ_{\ell}^s}\sum_{\substack{\bda \in \Z^n}}w_{\bdj/j_s}^*(\bda/j_s) e\big(dj_s f^*_{\bdj/j_s}(\ba/j_s)\big).
\end{equation*}
Summing over $1 \leq \ell \leq \ceil{\log_2 Q_*}$, it then follows that
\begin{equation*}\label{eq: self impr 1 3}
    \fkN^{*, s}_{\cM, w} (Q_*, \delta_*) \ll \delta_* Q_*^{n+R} + \Sigma^{*,s}(Q_*, \delta_*)
\end{equation*}
where
\begin{equation*}
    \Sigma^{*, s}(Q_*, \delta_*) := \delta_*  \sum_{\ell = 1}^{\ceil{\log_2 Q_*}} \sum_{d=1}^D \Big|\sum_{\bdj \in \cJ_{\ell}^s}\sum_{\substack{\bda \in \Z^n}}w_{\bdj/j_s}^*(\bda/j_s) e\big(dj_s f^*_{\bdj/j_s}(\ba/j_s)\big)\Big|.
\end{equation*}

By the Poisson summation formula and a change of variables,
\begin{equation*}
  \Sigma^{*, s}(Q_*, \delta_*) = \delta_* \sum_{\ell = 1}^{\ceil{\log_2 Q_*}}  \sum_{d=1}^D \Big|\sum_{\bdj \in \cJ_{\ell}^s} j_s^n \sum_{\substack{\bdk \in \Z^n}}\cI^{*,s}(d, \bdj, \bdk) \Big|  
\end{equation*}
where 
\begin{equation*}
  \cI^{*,s}(d, \bdj, \bdk) := \int_{\R^n}  w_{\bdj/j_s}^*(\bdy) e\big(-d j_s \phi_{\bdj/j_s}^*(\bdy; \bdk/d))\big)\,\ud \bdy \quad \textrm{for} \quad \phi_{\bdtheta}^*(\bdy; \bdz) := \bdy\cdot\bdz - f^*_{\bdtheta}(\bdy).
\end{equation*}

\noindent \textit{Step 2: Stationary phase}. We next evaluate $\cI^{*,s}(d, \bdj, \bdk)$ using the method of stationary phase. Recalling Definition~\ref{dfn: weight ext}, since $\widetilde{w} \in \cW_{\cM}(w; h)$, there exist $0 < \varepsilon_1 < \varepsilon_2 \leq 2\varepsilon$ with $\varepsilon_2 - \varepsilon_1 \geq h > 0$ such that $\supp w \subseteq B(\bdx_0, \varepsilon_1)$ and $\widetilde{w}(\bdy) = 1$ for all $\bdy \in B(\bdx_0, \varepsilon_2)$. We consider two cases, defined in terms of the ball $B(\bdx_0, \varepsilon_2)$.\medskip

\noindent \underline{Case 1.} Non-stationary: $\bdk/d \notin B(\bdx_0, \varepsilon_2)$.\medskip

\noindent If $\bdy \in \supp w_{\bdtheta}^*$ for $\bdtheta \in \R^R$ with $1 \leq |\theta|_{\infty} \leq 2$, then
\begin{equation*}
  (\nabla_{\bdy} f_{\bdtheta}^*)(\bdy) = (\nabla_{\bdx}f_{\bdtheta})^{-1}(\bdy) \in \supp w \subseteq B(\bdx_0, \varepsilon_1).   
\end{equation*}
Thus, 
\begin{equation*}
    |\nabla_{\bdy} \phi_{\bdtheta}^*(\bdy; \bdk/d)| \geq |\bdk/d -\bdx_0| - |(\nabla_{\bdy} f_{\bdtheta}^*)(\bdy) - \bdx_0| \geq \varepsilon_2 - \varepsilon_1 \geq h > 0.
\end{equation*}
Hence,
\begin{equation*}
    |\nabla_{\bdy} d j_s \phi_{\bdj/j_s}^*(\bdy; \bdk/d)| \gg_h d|j_s| \qquad \textrm{for all $\bdy \in \supp w_{\bdj/j_s}^*$.}
\end{equation*}
On the other hand, if $|\bdk| \geq Cd$ for a suitably large constant $C\geq 1$, then
\begin{equation*}
    |\nabla_{\bdy} d j_s\phi_{\bdj/j_s}^*(\bdy; \bdk/d)| \gg |\bdk|. 
\end{equation*}
Thus, by repeated integration-by-parts,
\begin{equation*}
    |\cI^{*, s}(d, \bdj, \bdk)| \ll_{h, \rho, M, N} |j_s|^{-N}|d|^{-N}|\bdk/d|^{-M} \qquad \textrm{for all $M$, $N \in \N_0$}
\end{equation*}
and the contribution in this case is essentially negligible. \medskip 

\noindent \underline{Case 2.} Stationary: $\bdk/d \in B(\bdx_0, \varepsilon_2)$.\medskip

\noindent Let $\bdtheta \in \R^R$ satisfy $1 \leq |\bdtheta|_{\infty} \leq 2$ and, recalling \eqref{eq: L trans 1}, observe that
\begin{equation*}
    \nabla_{\bdy}\phi_{\bdtheta}^*(\bdy; \bdz) = \bdz - (\nabla_{\bdy} f_{\bdtheta}^*)(\bdy) = \bdz - (\nabla_{\bdx} f_{\bdtheta})^{-1}(\bdy).
\end{equation*}
Further recall that $\nabla_{\bdx} f_{\bdtheta} \colon \cD \to \cD_{\bdtheta}^*$ is a diffeomorphism. It follows that for fixed $\bdz \in \cD$, the phase $\bdy \mapsto \phi_{\bdtheta}^*(\bdy; \bdz)$ has a unique critical point in $\cD_{\bdtheta}^*$ at the point $\bdy(\bdz) := (\nabla_{\bdx} f_{\bdtheta})(\bdz)$. Moreover, by \eqref{eq: L trans 2}, the Hessian matrix with respect to $\bdy$ satisfies
\begin{equation*}
    (\partial_{\bdy \bdy}^2 \phi_{\bdtheta}^*)(\bdy(\bdz), \bdz) = - (f_{\bdtheta}'')^{-1}(\bdz),
\end{equation*}
where the inverse is well-defined by our hypothesis \eqref{eq: CC}. In particular, the critical point is non-degenerate. Finally, evaluating the phase at the critical point yields
\begin{equation*}
    \phi_{\bdtheta}^*(\bdy(\bdz); \bdz) = (\nabla_{\bdx} f_{\bdtheta}^*)^{-1}(\bdz) \cdot \bdz - f_{\bdtheta}^*\circ (\nabla_{\bdx} f_{\bdtheta}^*)^{-1}(\bdz) = f_{\bdtheta}^{**}(\bdz) = f_{\bdtheta}(\bdz),
\end{equation*}
where we have used the definition of the Legendre transform and the remaining identity in \eqref{eq: L trans 1}. 

Fix $1 \leq d \leq D$, $\bdj \in \cJ_{\ell}^s$ and $\bdk \in \Z^n$ satisfying $\bdk/d \in B(\bdx_0, \varepsilon_2)$. Since, by hypothesis $\varepsilon_2 \leq 2\varepsilon_0 = \tau^2$, it follows from \eqref{eq: codim R ift 1} that 
\begin{equation*}
   B(\bdx_0, \varepsilon_2) \subseteq B(\bdx_0, \tau^2) \subseteq \cD. 
\end{equation*}
 Using the observations of the previous paragraph, we may apply the method of stationary phase in the form of Lemma~\ref{lem: stationary phase} to evaluate $\cI^{*,s}(d, \bdj, \bdk)$. In particular,
\begin{equation}
    \label{eq: self impr 2 1}
    \cI^{*,s}(d, \bdj, \bdk) = (d j_s)^{-n/2}e\big(- d j_s f_{\bdj/j_s}(\bdk/d)\big) a_{d, \bdk}^s(\bdj)
\end{equation}
for 
\begin{equation*}
a_{d, \bdk}^s(\bdu) :=  a(d u_s; \bdu/u_s, \bdk/d) \qquad \textrm{for $\bdu = (u_1, \dots, u_R) \in \R^R$, $u_s \neq 0$,}
\end{equation*}
where $a \in C^{\infty}\left([1,\infty) \times [0, 2]^R\times \cD\right)$ satisfies 
  \begin{equation}\label{eq: self impr 2 2}
     \sup_{\mathbf{t} \in [0, 2]^R}  \sup_{\bx \in \cD} |\partial_{\lambda}^{\iota} \partial^{\bdalpha}_{\mathbf{t}} a(\lambda;\mathbf{t}; \bdx)| \ll_{\rho} \lambda^{-\iota} \qquad \textrm{for all $\bdalpha \in \{0,1\}^R$, $\iota \in \{0,1\}$. } 
  \end{equation}
The expression for the phase in \eqref{eq: self impr 2 1} may be simplified by noting
\begin{equation*}
    j_s f_{\bdj/j_s}(\bdx) = j_s f(\bdx) \cdot \bdj/j_s = f(\bdx) \cdot \bdj.
\end{equation*}

In light of the above, we may write
\begin{multline*}
    \Sigma^{*, s}(Q_*, \delta_*) =  \\ \delta_* \sum_{\ell = 1}^{\ceil{\log_2 Q_*}} \sum_{d=1}^D \Big|\sum_{\bdj \in \cJ_{\ell}^s} j_s^n \sum_{\substack{\bdk \in \Z^n \\ \bdk/d \in B(\bdx_0, \varepsilon_2)}}  (d j_s)^{-n/2}e(- d f(\bdk/d) \cdot \bdj) a_{d, \bdk}^s(\bdj)\Big|
    + O_{h, \rho}(\delta_*),
\end{multline*}
where the $O_{h, \rho}(\delta_*)$ error term arises from the non-stationary terms from Case 1. Recall that $\widetilde{w}$ satisfies $\widetilde{w}(\bdx) = 1\geq w(\bx)$ for all $\bdx \in B(\bdx_0, \varepsilon_2)$. Thus, after rearranging the order of summation, we may dominate
\begin{equation}\label{eq: self impr 2 3}
    \Sigma^{*, s}(Q_*, \delta_*) \leq \Sigma_0^{*, s}(Q_*, \delta_*) + O_{h, \rho}(\delta_*)
\end{equation}
where
\begin{equation}\label{eq: self impr 2 4}
    \Sigma_0^{*, s}(Q_*, \delta_*) := \delta_*\sum_{d=1}^D d^{-n/2} \sum_{\bdk \in \Z^n} \widetilde{w}(\bdk/d) \sum_{\ell = 1}^{\ceil{\log_2 Q_*}} \Big|\sum_{\bdj \in \cJ_{\ell}^s} j_s^{n/2}e(- d f(\bdk/d) \cdot \bdj) a_{d, \bdk}^s(\bdj)\Big|.
\end{equation}
We shall relate the exponential sum appearing in \eqref{eq: self impr 2 4} to the counting function appearing in our hypothesis \eqref{eq: codim R hypothesis}.\medskip

\noindent \textit{Step 3: Passing to the original counting problem}. We estimate the sum in $\bdj$ using summation-by-parts as in Lemma~\ref{lem: summation by parts}. To this end, define
\begin{equation*}
    \fka_{d, \bdk}^s(\bdu) := u_s^{-n/2} a_{d,\bdk}^s(\bdu). 
\end{equation*}
For $1 \leq \ell \leq \ceil{\log_2 Q_*}$, recalling the definitions from \S\ref{subsec: sum by parts}, the set
\begin{equation*}
   \bdJ_{\ell}^s := \{(u_1, \dots, u_R ) \in [0, 2^{\ell+1}]^R \textrm{ with } 2^{\ell-1} \leq u_s \leq 2^{\ell} \big\} 
\end{equation*}
satisfies $\bdJ_{\ell}^s \in \fkJ^R(2^{\ell + 1})$ and $\cJ_{\ell}^s \subseteq \bdJ_{\ell}^s$. Moreover, it follows from \eqref{eq: self impr 2 2} that
\begin{align*}
 &\sum_{\bdalpha \in \{0, 1\}^R} 2^{(\ell + 1) |\bdalpha|}\sup_{\bdu \in \bdJ_{\ell}^s} \big|\partial_{\bdu}^{\bdalpha} \fka_{d, \bdk}^s(\bdu)\big| \ll_{\rho} 2^{\ell n/2}, 
\end{align*}
yielding an estimate of the form \eqref{eq sbp wt cond}.
Thus, we have $\fka_{d, \bdk}^s \in \fkA(2^{\ell+1}; \bdJ_{\ell}^s; C_{\theta}2^{\ell n/2})$ for some $C_{\theta}>1$. Using Lemma \ref{lem: summation by parts}, we thereby conclude that 
\begin{equation*}
    \Big|\sum_{\bdj\in \cJ_{\ell}^s}j_s^{n/2}e\big(-d f(\bk/d)\cdot \bdj \big)a_{d, \bdk}^s(\bdj)\Big| \ll_{\rho} 2^{\ell n/2} \prod_{r = 1}^R \min\big\{2^\ell, \|df_r\left(\bk/d\right)\|^{-1}\big\}.
\end{equation*} 
Summing over $1 \leq \ell \leq \ceil{\log_2 Q_*}$ and applying the resulting estimate to \eqref{eq: self impr 2 4} yields
\begin{equation}\label{eq: self impr 3 1}
 \Sigma_0^{*, s}(Q_*, \delta_*) \ll_{\rho} \delta_* Q_*^{n/2}\sum_{d=1}^D d^{-n/2}\sum_{\bdk\in \Z^n}{\widetilde{w}(\bdk/d)}\prod_{r=1}^R\min\big\{Q_*, \|df_r(\bk/d)\|^{-1}\big\}.
\end{equation}
We split the sum in \eqref{eq: self impr 3 1} according to the relative sizes of $\|df_r(\bk/d)\|$ and $Q_*$  for $1 \leq r \leq R$, leading to the bound 
\begin{equation}
\label{eq: self impr 3 2}
\Sigma_0^{*, s}(Q_*, \delta_*) \ll_{\rho} \delta_* Q_*^{n/2 + R}\sum_{\substack{\bdi = (i_1, \dots, i_r) \in \N_0^R\\1 \leq 2^{i_r} \leq Q_* \\ 1 \leq r \leq R}}2^{-\sum_{r=1}^Ri_r} \Big[\sum_{d=1}^D d^{-n/2} \sum_{\substack{\bk\in \mathbb{Z}^n\\ \|df_r(\bk/d)\|\leq 2^{i_r}Q_*^{-1} \\1\leq r\leq R}} \widetilde{w}(\bdk/d)\Big].
\end{equation}
Applying a dyadic decomposition and the definition \eqref{eq: codim R smooth count} of the counting function, the bracketed expression on the right-hand side of \eqref{eq: self impr 3 2} may be bounded by 
\begin{equation}\label{eq: self impr 3 3}
    \sigma_{\bdi}(Q_*, \delta_*) := \sum_{\substack{1 \leq \lambda \leq D \\ \mathrm{dyadic}}} \lambda^{-n/2} \fkN_{\cM, \widetilde{w}}(\lambda, \bddelta_{\bdi})
\end{equation}
where $\bddelta_{\bdi} = (2^{i_1}Q_*^{-1}, \dots, 2^{i_R} Q_*^{-1})$ for $\bdi = (i_1, \dots, i_R) \in \N_0^R$ satisfying $1 \leq 2^{i_r} \leq Q_*$ for $1 \leq r \leq R$.\medskip

\noindent \textit{Step 4: Applying the hypothesised bounds}.  Observe that
\begin{equation*}
    \bddelta_{\bdi}^{\times} = 2^{\sum_{r=1}^R i_r} Q_*^{-R} \qquad \textrm{and} \qquad \bddelta_{\bdi}^{\times}(\bddelta_{\bdi}^{\wedge})^{-R} = 2^{\sum_{r=1}^R i_r - R\min\{i_1, \dots, i_R\}}.
\end{equation*}
We now apply the hypothesised bound \eqref{eq: codim R hypothesis} to estimate each term in \eqref{eq: self impr 3 3}, giving
\begin{equation*}
 \fkN_{\cM, \widetilde{w}}(\lambda, \bddelta_{\bdi}) \leq   A\big( 2^{\sum_{r = 1}^R i_r} Q_*^{-R} \lambda^{n + 1}  + 2^{\sum_{r=1}^R i_r - R\min\{i_1, \dots, i_R\}} \lambda^{\beta} \log^{\gamma} \lambda\big).
\end{equation*}
Since $\beta > n/2$, multiplying through by $\lambda^{-n/2}$, summing over $1 \leq \lambda \leq D$ dyadic and recalling that $D \sim \delta_*^{-1}$ yields
\begin{equation*}
    \sigma_{\bdi}(Q_*, \delta_*) \ll A\big(2^{\sum_{r = 1}^R i_r} \delta_*^{-n/2 - 1} Q_*^{-R} + 2^{\sum_{r=1}^R i_r - R\min\{i_1, \dots, i_R\}} \delta_*^{-\beta + n/2}|\log\delta_*|^{\gamma}\big).
\end{equation*}
The above inequality can be applied to bound \eqref{eq: self impr 3 2}, giving
\begin{align}
\label{eq: self impr 4 1}
\Sigma_0^{*, s}(Q_*, \delta_*) &\ll_{\rho} A\big(\delta_*^{-n/2}  Q_*^{n/2} \log^R Q_* + \delta_*^{-\beta+ n/2 + 1} |\log\delta_*|^{\gamma} Q_*^{n/2 + R} \log^{R-1} Q_* \big).
\end{align}
Note that the $O(\delta_*)$ error in \eqref{eq: self impr 2 3} is bounded by the right-hand side of \eqref{eq: self impr 4 1}. Thus, combining \eqref{eq: self impr 4 1} with \eqref{eq: self impr 1 3} and \eqref{eq: self impr 2 3}, we obtain
\begin{align*}
&\fkN^{*, s}_{\cM, w}  (Q_*, \delta_*) \ll_{\rho, h} \\
&A \Big(\,\,\underbrace{\delta_* Q_*^{n+R}}_{\tO \, :=} + \underbrace{\delta_*^{-n/2}  Q_*^{n/2} \log^R Q_*}_{\tI \,:= } + \underbrace{\delta_*^{-\beta+ n/2 + 1}|\log\delta_*|^{\gamma} Q_*^{n/2 + R} \log^{R-1} Q_*}_{\tII \,:= } \,\,\Big),
\end{align*}
It remains to compare the terms appearing in the above bound in order to arrive at our final estimate.\medskip

\noindent \textit{Step 5: Comparing terms}. We make the following observations:
\begin{equation*}
\tO \leq \tI \quad \Longleftrightarrow \quad \delta_{*} \leq \delta_{\tI}  := Q_*^{- \frac{n+2R}{n+2}} \, \log^{\frac{2R}{n+2}} Q_*,    
\end{equation*}
and
\begin{equation*}
  \tO \leq \tII \quad \Longleftrightarrow \quad \delta_{*} \leq \delta_{\tII} \quad \textrm{where}\quad
 \delta_{\tII}^{\beta - \frac{n}{2}}|\log\delta_{\tII}|^{-\gamma} = Q_*^{-n/2} \log^{R-1} Q_*. 
\end{equation*}
The latter condition implies that $|\log \delta_{\tII} | \sim \log Q_*$. Hence 
\begin{equation}\label{eq: self impr 5 1}
    \delta_{\tII} \sim Q_*^{-\frac{n}{2\beta - n}} \log^{\frac{2(R-1 + \gamma)}{2\beta - n}} Q_*. 
\end{equation}
Furthermore, the hypothesis $\beta > 
\frac{n(n+R+1)}{n+2R}$ ensures that $\delta_{\tI} \ll_{\beta} \delta_{\tII}$  (see Remark~\ref{rmk: alpha beta}).

If $\delta_* \geq \max\{\delta_{\tI}, \delta_{\tII}\}$, then the above observations imply $\tO \geq \max \{\tI, \tII\}$ and so
\begin{equation*}
   \fkN_{\cM, w}^{*, s}(Q_*,\delta_*)  \ll_{\rho, h} A\delta_* Q_*^{n+R},
\end{equation*}
which is a favourable estimate. On the other hand, if $ \delta_* \leq \max\{\delta_{\tI}, \delta_{\tII}\}$, then we can apply the monotonicity of the counting function to deduce that 
\begin{equation*}
 \fkN_{\cM, w}^{*, s}(Q_*,\delta_*) \leq \fkN_{\cM, w}^{*, s}(Q_*, \eta_*) \quad \textrm{where} \quad \eta_* := \max\{\delta_{\tI}, \delta_{\tII}\}\,.
\end{equation*}
Since $\eta_* = \max\{\delta_{\tI}, \delta_{\tII}\}$, the term $\tO$ dominates both the terms $\tI$ and $\tII$ when evaluated at $\eta_*$. 
Therefore, in light of \eqref{eq: self impr 5 1}, we deduce that
\begin{equation*}
 \fkN_{\cM, w}^{*, s}(Q_*,\delta_*)  \ll_{\rho, h} A \eta_* Q_*^{n+R}  \ll_{\beta} A Q_*^{n + R-n/(2\beta - n)} \log^{\frac{2(R-1 + \gamma)}{2\beta - n}} Q_*.
\end{equation*}
Thus, in all cases we obtain a desirable bound. 
\end{proof}




\subsection{Iteration scheme} Here we show that repeated application of the self-improving mechanism from Lemma~\ref{lem: codim R self imp} and Lemma~\ref{lem: codim R self imp *} leads to Theorem~\ref{thm: aniso refined}.

\begin{proof}[Proof (of Theorem~\ref{thm: aniso refined})] The proof proceeds in the same way as the proof of the $R\geq 3$ case of \cite[Theorem 1.9]{Srivastava2025}.

Fix $w \in \cW_{\cM}$ satisfying $\supp w \subseteq B(\bdx_0, \varepsilon)$ and $w(\bdy) = 1$ for $\bdy \in B(\bdx_0, \varepsilon/2)$, where $\varepsilon := \tau^2/2$ as above. Let $\nu > 0$ be given. As discussed in \S\ref{subsec: prelim}, it then suffices to show
\begin{equation}\label{eq: codim R pf 1}
    \fkN_{\cM, w}(Q, \bddelta) \ll_{\nu}  \bddelta^{\times}Q^{n+1} + \bddelta^{\times} (\bddelta^{\wedge})^{-R} Q^{n + 1 - \frac{(n+2)R}{n+2R} + \nu}.
\end{equation}

 The strategy is to successively apply Lemma~\ref{lem: codim R self imp} and Lemma~\ref{lem: codim R self imp *} to obtain bounds for our counting function. Unfortunately, there is a minor technicality regarding the choice of weights. To address this, let $N = N(\nu) \in \N$ satisfy
\begin{equation}
    \label{eq: codim R pf 2}
    \frac{\nu}{4}\leq \left(\frac{4}{5}\right)^{N(\nu)}\, R<  \frac{\nu}{2}
\end{equation}
and fix $w_0 := w$ and $\varepsilon_0 := \varepsilon$. Let $\omega \colon \R^n\to [0,1]$ be a smooth function satisfying
\begin{equation*}
    \supp \omega \subseteq B(\bzero, 1) \qquad \textrm{and} \qquad \int_{\R^n} \omega=1
\end{equation*}
For $k\in \N_0$, define and
\begin{equation*}
 \omega_k(\bx):=\varepsilon_0^{-n}2^{kn}\omega(\varepsilon_0^{-1}2^k\bx) \quad \textrm{and} \quad   \varepsilon_k:={\varepsilon_0}\sum_{j=0}^k 2^{-j} \in [\varepsilon_0, 2\varepsilon_0).
\end{equation*}
For $1\leq k\leq N$, we define a sequence of smooth weights $w_k \colon \R^n \to [0, 1]$ by
\begin{equation*}
    w_k(\bx) := \mathbbm{1}_{B(\bx_0,\, \varepsilon_k+2^{-k-2}\varepsilon_0)}* \omega_{k+2}(\bx)\,.
\end{equation*}
It is straightforward to check that
\begin{equation*}
    \supp\, w_k\subset B(\bx_0, \varepsilon_{k+1}) \qquad \textrm{ and } \qquad w_k(\bx)=1 \textrm{ for all } \bx\in B(\bx_0, \varepsilon_{k})
\end{equation*}
for each $1\leq k\leq N(\nu)$. Furthermore, for all $\bx \in B(\bx_0, \varepsilon_{k+1})$ and $\bdalpha\in \N^n$, we have
 \begin{equation*}
     |\partial^{\bdalpha}w_k(\bx)|\ll_{\bdalpha} \big(2^{-k}\varepsilon_0\big)^{-|\bdalpha|}\leq \big(2^{-N}\varepsilon_0\big)^{-|\bdalpha|}\ll (\nu^c\,\varepsilon_0)^{-|\bdalpha|},
    \end{equation*}
where $c := \frac{\log\, 2}{\log\, 5/4}$ and in the last inequality we used \eqref{eq: codim R pf 2}.
This ensures that, in the language of Definition~\ref{dfn: weight ext}, there exists some $0 < h = h(\nu) \leq 1$ and $\rho = \rho(\nu) \geq 1$ such that 
\begin{equation}\label{eq: codim R pf 3}
 w_k \in \cW_{\cM}(\rho) \quad \textrm{and} \quad   w_k\in \cW_{\cM}(w_{k-1}, h) \quad \textrm{for all $1\leq k\leq N$.}
\end{equation}

Trivially, $\fkN_{\cM, w_N}(Q, \bddelta) \ll Q^{n+1}$ holds for all $\bddelta \in (0, 1/4)^R$. Thus, provided $A \geq 1$ are suitably chosen,
\begin{equation*}
\fkN_{\cM, w_N}(Q, \bddelta) \leq A\big( \bddelta^{\times}Q^{n+1} +  \bddelta^{\times} (\bddelta^{\wedge})^{-R} Q^{\beta} \log^{\gamma} Q\big) 
\end{equation*}
holds for $\beta = n+1$ and $\gamma = 0$. We may then alternatively apply Lemma~\ref{lem: codim R self imp} and Lemma~\ref{lem: codim R self imp} to deduce better and better bounds for $\fkN_{\cM, w}(Q, \bddelta)$. In particular, if we recursively define
\begin{align*}
    \beta_0 &:= n+1,& \qquad \qquad \beta_{j+1} &:= n + 1 - \frac{nR}{2 \alpha_{j+1} - n} \\
    \alpha_{j+1} &:= n + R - \frac{n}{2 \beta_j - n}, & \qquad \qquad &=n+1 - \frac{Rn}{2R +n(1- \frac{2}{2\beta_j - n})},
\end{align*}
then the observations in Remark~\ref{rmk: alpha beta} ensure that 
\begin{equation*}
   \frac{n(n+R+1)}{n+2R} < \beta_j \leq n+1 \qquad \textrm{and} \qquad\frac{n(n+R+1)}{n+2} < \alpha_j \leq n +R
\end{equation*}
 for all $j \in \N_0$. By the argument described above, we conclude that there exists some constant $C_{\nu} \geq 1$ and a sequence of nonnegative exponents $(\gamma_j)_{j=0}^N$ such that 
\begin{equation}\label{eq: codim R pf 4}
\fkN_{\cM, w_{N-j}}(Q, \bddelta) \leq C_{\nu}^jA \big(\bddelta^{\times}Q^{n+1} +  \bddelta^{\times} (\bddelta^{\wedge})^{-R} Q^{\beta_j} \log^{\gamma_j} Q\big)  \qquad \textrm{for $0 \leq j \leq N$.} 
\end{equation}
Here, $\gamma_j$ depends only on $n$, $R$ and $j$. Note that the choice of weights, and in particular \eqref{eq: codim R pf 3}, allow us to ensure that the hypotheses of Lemma~\ref{lem: codim R self imp} and Lemma~\ref{lem: codim R self imp *} are satisfied at each stage.

It was shown in \cite[pp. 19--21]{Srivastava2025} that 
\begin{equation*}
    \beta_N - n - 1 + \frac{(n+2)R}{n+2R}\leq \left(\frac{4}{5}\right)^N R < \frac{\nu}{2},
\end{equation*}
where the last step is due to our choice of $N$ from \eqref{eq: codim R pf 2}. Furthermore, since $N$ depends only on $\nu$ and $R$, it follows that
\begin{equation*}
    C_{\nu}^N \ll_{\nu} 1 \qquad \textrm{and} \qquad \log^{\gamma_N} Q \ll_{\nu} Q^{\nu/2}.
\end{equation*}
Consequently, the $j = N$ case of \eqref{eq: codim R pf 4} combined with the above observations yields \eqref{eq: codim R pf 1}. This establishes Theorem \ref{thm: aniso refined}.
\end{proof}




\section{The curvature condition revisited}\label{sec: examples}

Here we develop the theory of the conditions (CC) and ($\C$-CC) by proving Lemma~\ref{lem: ND vs CND} and deriving an implicit formulation of the curvature condition. We then discuss the details of the specific case of the complex sphere, as introduced in Example~\ref{ex: sphere} and Example~\ref{ex: sphere Gauss} above.




\subsection{Real vs complex curvature conditions} We begin by proving the identity from Lemma~\ref{lem: ND vs CND}, relating the conditions (CC) and ($\C$-CC). 

\begin{proof}[Proof (of Lemma~\ref{lem: ND vs CND})] For $1 \leq j \leq 2m$, let $c_j$ and $d_j$ denote the $j$th column of $u''$ and $v''$, respectively, so that
\begin{equation*}
  u''  = 
    \begin{pmatrix}
        c_1 & \cdots & c_m & c_{m+1} & \cdots & c_{2m}
    \end{pmatrix}, \quad
    v'' =
    \begin{pmatrix}
        d_1 & \cdots & d_m & d_{m+1} & \cdots & d_{2m}
    \end{pmatrix}  .
\end{equation*}
Recall that the Cauchy--Riemann equations state that $u_{x_j} = v_{y_j}$ and $u_{y_j} = - u_{x_j}$ for all $1 \leq j \leq m$. Using this, we see that
\begin{equation}\label{eq: ND vs CND 1}
    v'' =
    \begin{pmatrix}
        -c_{m+1} & \cdots & -c_{2m} & c_1 & \cdots & c_m
    \end{pmatrix}.
\end{equation}

By the multilinearity of the determinant,
\begin{equation}\label{eq: ND vs CND 1.5}
\det \bigl(\theta_1 u'' + \theta_2 v''\bigr) = \sum_{S\subseteq \{1, \dots, 2m\}} \theta_1^{2m - \#S} \theta_2^{\#S} \det M_S    
\end{equation}
where $M_S$ is the $2m \times 2m$ matrix whose $j$th column is equal to $c_j$ if $j \notin S$ and is equal to $d_j$ if $j \in S$ for $1 \leq s \leq 2m$. Furthermore, using \eqref{eq: ND vs CND 1}, we see that $\det M_S = 0$ unless $S$ is of the form 
\begin{equation}\label{eq: ND vs CND 2}
 S = X \cup \{j + m : j \in X \} \qquad \textrm{for some $X \subseteq \{1, \dots, m\}$.}
\end{equation}
Indeed, if this is not the case, then it follows from \eqref{eq: ND vs CND 1} that two of the columns of $M_S$ agree up to a sign change. On the other hand, if $S$ is of the form \eqref{eq: ND vs CND 2}, then $\det M_S = \det u''$. Indeed, in this case, the columns of $M_S$ are formed by a permutation of the columns of $u''$ of signature $\#X$ (to form $M_S$ from $u''$ we transpose $\#X$ columns) and changing the sign of $\#X$ many columns. The combined effect of the permuting the columns and changing the sign of the columns is that the determinant changes sign $2\#X$ times. Combining these observations with \eqref{eq: ND vs CND 1.5}, we conclude that 
\begin{equation}\label{eq: ND vs CND 3}
 \det \bigl(\theta_1 u'' + \theta_2 v''\bigr) = \sum_{X \subseteq \{1,\dots, m\}} \theta_1^{2(m - \#X)} \theta_2^{2\#X} \det  u'' =   
(\theta_1^2 + \theta_2^2)^m \det  u''.
\end{equation}  

Now consider the holomorphic map $F \colon U \to \C^m$ given by 
\begin{equation*}
F := (\partial_{\bdz} f) = (\partial_{z_1} f, \ldots, \partial_{z_m} f).  
\end{equation*}
We can identify this with a real map $F_{\R} \colon U_{\R} \to \R^{2m}$ which, by the Cauchy--Riemann equations, has components
\begin{equation*}
F_{\R} =  (u_{x_1}, - u_{y_1}, \ldots, u_{x_n}, -u_{y_n})
\end{equation*}
Then the complex Jacobian determinant $JF$ of $F$ and the real Jacobian determinant $J_{\R}F$ of $F_{\R}$ satisfy
\begin{equation*}
    JF = \det \partial_{\bdz \bdz}f \qquad \textrm{and} \qquad J_{\R}F = (-1)^m\det u''
\end{equation*}
By, for instance, \cite[p.11, eq. (8)]{Rudin_book},  the above real and complex Jacobian determinants are related by the formula $J_{\R}F(\bdx, \bdy) = |JF(\bdx + i\bdy)|^2$. Thus,
\begin{equation}\label{eq: ND vs CND 4}
    (-1)^m \det  u''(\bdx, \bdy) = \ \bigl| \det \partial_{\bdz \bdz} f(\bdx + i \bdy) \bigr|^2.
\end{equation} 
Hence, combining \eqref{eq: ND vs CND 3} and \eqref{eq: ND vs CND 4}, we obtain
\begin{equation*}
 \det \bigl(\theta_1 u''(\bdx, \bdy) + \theta_2 v''(\bdx, \bdy) \bigr) = 
(-1)^m \, (\theta_1^2 + \theta_2^2)^m \ \bigl| \det \partial_{\bdz \bdz}f(\bdx + i \bdy) \bigr|^2   
\end{equation*}
for $(\bdx, \bdy) \in U_{\R}$, as required.
\end{proof}




\subsection{An implicit description of the curvature condition}\label{subsec: implicit CC}

Let $n$, $R \in \N$ and $\Phi = (\Phi_1, \dots, \Phi_R) \colon \R^{n + R}  \to \R^R$ be a smooth function. Suppose $\Phi$ is regular in the sense that 
\begin{equation*}
    \mathrm{rank} \,\partial_{\bdu} \Phi(\bdu) = R \qquad \textrm{wherever $\bdu \in \R^{n+R}$ satisfies} \qquad \Phi(\bdu) = \bdzero.  
\end{equation*}
By the implicit function theorem, $\cM := \{\bdu \in \R^{n+R} : \Phi(\bdu) = \bdzero\}$ is a smooth submanifold of $\R^{n+R}$ of dimension $n$, at least provided it is nonempty.

Fix $\bdu_0 = (\bdx_0, \bdy_0) \in \cM \subset \R^n \times \R^R$ and, by relabelling coordinates, assume that 
\begin{equation*}
\det \partial_{\bdy} \Phi(\bdx, \bdy)|_{(\bdx, \bdy) = (\bdx_0, \bdy_0)} \neq 0.    
\end{equation*}
In a neighbourhood of $\bdu_0$, the manifold $\cM$ admits a graph parametrisation as in \eqref{eq: real graph} of some smooth function $f = (f_1, \dots, f_R) \colon U \to \R^R$, where $U \subseteq \R^n$ is an open neighbourhood of $\bdx_0$, and
\begin{equation}\label{eq: regular map}
    \det \partial_{\bdy} \Phi(\bdx, \bdy) \neq 0 \quad \textrm{for all $(\bdx, \bdy) \in U$} \qquad \textrm{and} \qquad f(\bdx_0) = \bdy_0.
\end{equation}

The goal of this subsection is to relate the condition \eqref{eq: CC} expressed in terms of the graphing function $f$ to a condition on $\Phi$. To this end, given $\bdt \in \R^R$, we define
\begin{equation*}
    \bdt \cdot f'' := \sum_{\ell = 1}^R t_{\ell} f_{\ell}'' \in \mathrm{Mat}_{n \times n}(\R) \quad \textrm{and} \quad \bdt \cdot \partial_{\bdu\bdu}^2\Phi := \sum_{\ell = 1}^R t_{\ell} \partial_{\bdu\bdu}^2\Phi_{\ell} \in \mathrm{Mat}_{(n+R) \times (n+R)}(\R),
\end{equation*}
where $\partial_{\bdu\bdu}^2\Phi_{\ell}$ denotes the hessian determinant of $\Phi_{\ell}$ with respect to all $n + R$ variables. We now give an implicit formulation of the curvature condition (CC).

\begin{lemma}\label{lem: implicit CC} With the above setup, for $\bds \in \R^R$ and $\bdt := (\partial_{\bdy} \Phi)(\bdx, f(\bdx))^{\top} \bds$, we have
\begin{equation}\label{eq: implict CC}
    |\det \partial_{\bdy} \Phi(\bdx, f(\bdx))|^2|\det \bdt \cdot f''(\bdx)| = \left| \det
    \begin{bmatrix}
        \bds \cdot \partial_{\bdu \bdu}^2 \Phi(\bdx, f(\bdx)) & \partial_{\bdu} \Phi(\bdx, f(\bdx))^{\top} \\
        \partial_{\bdu} \Phi(\bdx, f(\bdx)) & \bdzero
    \end{bmatrix}
    \right| 
\end{equation}
for all $\bdx \in U$. 
\end{lemma}

By \eqref{eq: regular map}, we therefore see that the implicit form of condition (CC) is
\begin{equation*}
    \det
    \begin{bmatrix}
        \sum_{\ell = 1}^R\theta_{\ell} \partial_{\bdu \bdu}^2 \Phi_{\ell}(\bdu) & \partial_{\bdu} \Phi(\bdu)^{\top} \\
        \partial_{\bdu} \Phi(\bdu) & \bdzero
    \end{bmatrix} \neq 0 \qquad
    \textrm{for all $\bdu \in \cM$, $(\theta_1, \dots, \theta_R) \in \bbS^{R-1}$.}
\end{equation*}

\begin{proof}[Proof (of Lemma~\ref{lem: implicit CC})] By implicit differentiation, the Jacobian $\partial_{\bdx} f$ satisfies
\begin{equation}\label{eq: implicit 1}
  \bdzero = \partial_{\bdx} \Phi(\bdx, f(\bdx)) + (\partial_{\bdy} \Phi)(\bdx, f(\bdx)) \partial_{\bdx} f(\bdx) \qquad \textrm{for $\bdx \in U$.}
\end{equation}
Henceforth, to ease notation, we shall write $\bdz = (\bdx, f(\bdx))$ for $\bdx \in U$. We fix $1 \leq \ell \leq R$ and differentiate the $\ell$th row of \eqref{eq: implicit 1} to deduce that
\begin{align*}
    \bdzero &= \partial_{\bdx \bdx}^2 \Phi_{\ell}(\bdx, f(\bdx)) + 2 \sum_{k=1}^R (\partial_{\bdx} \partial_{y_k} \Phi_{\ell})(\bdx, f(\bdx)) \partial_{\bdx} f_k(\bdx) \\
    & \qquad + \sum_{j, k = 1}^R (\partial_{\bdx} f_j(\bdx))^{\top}(\partial_{y_jy_k}^2 \Phi_{\ell})(\bdx, f(\bdx)) \partial_{\bdx} f_k(\bdx) + \sum_{k=1}^R (\partial_{y_k} \Phi_{\ell})(\bdx, f(\bdx)) f_k''(\bdx)
\end{align*}
for $\bdx \in U$. Now let $\bds = (s_1, \dots, s_R) \in \R^R$ and $\bdt = (t_1, \dots, t_R) := (\partial_{\bdy} \Phi)(\bdx, f(\bdx))^{\top} \bds$. Multiplying both sides of the above equation by $t_{\ell}$, rearranging and summing in $\ell$, one deduces that
\begin{align}\label{eq: implicit 2}
    -  \bdt \cdot f''(\bdx) &= 
    (\bds \cdot \partial_{\bdx \bdx}^2 \Phi)(\bdx, f(\bdx))  + 2 (\bds \cdot \partial_{\bdx \bdy}^2 \Phi)(\bdx, f(\bdx)) \partial_{\bdx} f(\bdx) \\
    \nonumber 
    & \qquad + (\partial_{\bdx} f(\bdx))^{\top}( \bds \cdot \partial_{\bdy \bdy}^2 \Phi)(\bdx, f(\bdx)) \partial_{\bdx} f(\bdx)
\end{align}
where 
\begin{gather*}
\bds \cdot \partial_{\bdx \bdx}^2 \Phi := \sum_{\ell = 1}^R s_{\ell} \partial_{\bdx \bdx}^2 \Phi_{\ell}  \in \mathrm{Mat}_{n \times n}(\R), \quad
   \bds \cdot \partial_{\bdx \bdy}^2 \Phi := \sum_{\ell = 1}^R s_{\ell} \partial_{\bdx \bdy}^2 \Phi_{\ell}  \in \mathrm{Mat}_{n \times R}(\R), \\ 
   \bds \cdot \partial_{\bdy \bdy}^2 \Phi := \sum_{\ell = 1}^R s_{\ell} \partial_{\bdy \bdy}^2 \Phi_{\ell}  \in \mathrm{Mat}_{R \times R}(\R).
\end{gather*}

 In light of \eqref{eq: implicit 1}, we may replace each occurrence of $\partial_{\bdx} f$ in \eqref{eq: implicit 2} with $-(\partial_{\bdy} \Phi)^{-1} \partial_{\bdx} \Phi$. In particular, the right-hand side of \eqref{eq: implicit 2} is equal to 
 \begin{equation}\label{eq: implicit 3}
     (\bds \cdot \partial_{\bdx \bdx}^2 \Phi) + 2 (\bds \cdot \partial_{\bdx \bdy}^2 \Phi)(\partial_{\bdy} \Phi)^{-1} \partial_{\bdx} \Phi + (\partial_{\bdy} \Phi)^{-\top} (\partial_{\bdx}\Phi)^{\top}( \bds \cdot \partial_{\bdy \bdy}^2 \Phi) (\partial_{\bdy} \Phi)^{-1}\partial_{\bdx} \Phi
 \end{equation}
 evaluated at $\bdz$. Recall the block matrix identity 
 \begin{equation*}
     \det\begin{bmatrix}
         \bdA & \bdB \\
         \bdC & \bdD
     \end{bmatrix}
     = \det\begin{bmatrix}
        \bdA & \bdB \\
         \bdC & \bdD
     \end{bmatrix} \det\begin{bmatrix}
         \bdI & \bdzero \\
         -\bdD^{-1}\bdC & \bdI
     \end{bmatrix} =
    \det \bdD \det(\bdA - \bdB \bdD^{-1} \bdC)
 \end{equation*}
 for $\bdA \in \mathrm{Mat}_{n \times n}(\R)$, $\bdB \in \mathrm{Mat}_{n \times 2R}(\R)$, $\bdC \in \mathrm{Mat}_{2R \times n}(\R)$ and $\bdD \in \mathrm{GL}_{2R}(\R)$. We apply this
 with $\bdA := \bds \cdot \partial_{\bdx \bdx}^2 \Phi$ to deduce that the determinant of the matrix in \eqref{eq: implicit 3} is equal to
 \begin{equation*}
    (-1)^R\det \begin{bmatrix}
        \bds \cdot \partial_{\bdx \bdx}^2 \Phi & (\bds \cdot \partial_{\bdx \bdy}^2 \Phi) (\partial_{\bdy} \Phi)^{-1} & (\partial_{\bdx} \Phi)^{\top} \\
        (\partial_{\bdy} \Phi)^{-\top} (\bds \cdot \partial_{\bdy \bdx}^2 \Phi) & (\partial_{\bdy} \Phi)^{-\top} (\bds \cdot \partial_{\bdy \bdy}^2 \Phi) (\partial_{\bdy} \Phi)^{-1} & \bdI \\
        \partial_{\bdx} \Phi & \bdI & \bdzero
    \end{bmatrix}.
\end{equation*}
Furthermore, the matrix in the above display can be factored as
\begin{equation*}
     \begin{bmatrix}
       \bdI & \bdzero & \bdzero \\
       \bdzero &(\partial_{\bdy} \Phi)^{-\top} & \bdzero \\
       \bdzero & \bdzero & \bdI
     \end{bmatrix}
     \begin{bmatrix}
        \bds \cdot \partial_{\bdx \bdx}^2 \Phi & \bds \cdot \partial_{\bdx \bdy}^2 \Phi & (\partial_{\bdx} \Phi)^{\top} \\
         \bds \cdot \partial_{\bdy \bdx}^2 \Phi & \bds \cdot \partial_{\bdy \bdy}^2 \Phi & (\partial_{\bdz} \Phi)^{\top}  \\
        \partial_{\bdx} \Phi & \partial_{\bdy}  \Phi & \bdzero.
    \end{bmatrix}
    \begin{bmatrix}
       \bdI & \bdzero & \bdzero \\
       \bdzero &(\partial_{\bdy} \Phi)^{-1} & \bdzero \\
       \bdzero & \bdzero & \bdI
     \end{bmatrix}.
\end{equation*}
From this, recalling \eqref{eq: implicit 2},   we obtain the desired formula \eqref{eq: implict CC}. 
\end{proof}

The above argument also works in the complex setting. In particular, let $m \in \N$ and $\Phi \colon \C^{m+1} \to \C$ be a homomorphic function, which is regular in the sense that 
\begin{equation*}
    \partial_w \Phi(\bdz, w) \neq 0 \qquad \textrm{whenever $(\bdz, w) \in \C^m \times \C$ satisfies $\Phi(\bdz, w) = 0$.}
\end{equation*}
By the homomorphic implicit function theorem, $\cM := \{ \bdzeta \in \C^{m+1} : \Phi(\bdzeta) = 0\}$ is an analytic hypersurface in $\C^{m+1}$, at least provided it is nonempty. Moreover, $\cM$ can be locally parameterised as a graph as in \eqref{eq: analytic M}, where $U \subseteq \C^m$ is some open set and $f \colon U \to \C$ is a holomorphic function. Arguing as in the proof of Lemma~\ref{lem: implicit CC}, we have
\begin{equation*}
    |\partial_w \Phi(\bdz, f(\bdz))|^3 |\det f''(\bdz)| = \left| \det
    \begin{bmatrix}
        \partial_{\bdzeta \bdzeta}^2 \Phi(\bdz, f(\bdz)) & \partial_{\bdzeta} \Phi(\bdz, f(\bdz))^{\top} \\
        \partial_{\bdzeta} \Phi(\bdz, f(\bdz)) & 0
    \end{bmatrix}
    \right| 
\end{equation*}
for all $\bdz \in U$. We therefore see that the implicit form of condition ($\C$-CC) is
\begin{equation}\label{eq: cplx implicit CC}
    \det
    \begin{bmatrix}
     \partial_{\bdzeta \bdzeta}^2 \Phi(\bdzeta) & \partial_{\bdzeta} \Phi(\bdzeta)^{\top} \\
        \partial_{\bdzeta} \Phi(\bdzeta) & 0
    \end{bmatrix} \neq 0 \qquad
    \textrm{for all $\bdzeta \in \cM$.}
\end{equation}
We shall use this to study the complex sphere in the next subsection. 




\subsection{The complex sphere} We return to the complex sphere from Example~\ref{ex: sphere}, which we write as
\begin{equation*}
\bbS_{\C}^m := \big\{(\bdz, w) \in \C^m \times \C : \bdz \in B \textrm{ and } \Phi(\bdz, w) = 0 \big\}  
\end{equation*}
where the defining function $\Phi \colon \C^m \times \C \to \C$ is given by
\begin{equation*}
    \Phi(z_1, \dots, z_m, w) := z_1^2 + \cdots + z_m^2 + w^2 - 1 \qquad \textrm{for $(z_1, \dots, z_m, w) \in \C^m \times \C$.}
\end{equation*}
Note that, provided the radius $r$ of $B = B(0,r)$ is chosen sufficiently small, $\partial_w\Phi(\bdzeta) \neq 0$ for all $\bdzeta \in \bbS^m_{\C}$, and $\Phi$ is therefore regular as a complex defining function. Furthermore, up to a factor of $2^{2m+2}$, the matrix determinant in \eqref{eq: cplx implicit CC} is given by
\begin{equation*}
     \det\begin{bmatrix}
     \bdI & \bdzeta^{\top} \\
        \bdzeta & 0
    \end{bmatrix}
    = -(\zeta_1^2 + \cdots + \zeta_{m+1}^2) = -1 \neq 0
\end{equation*}
for all $\bdzeta = (\zeta_1, \dots, \zeta_{m+1}) \in \bbS_{\C}^m$. Thus, by the observations in the previous section, $\bbS_{\C}^m$ is an analytic hypersurface which satisfies ($\C$-CC).

As in the discussion in Example~\ref{ex: sphere},  we identify $\bbS_{\C}^m$ with a codimension $2$ real submanifold of $\R^{2m+2}$. By Lemma~\ref{lem: ND vs CND} and the preceding observations, this real submanifold satisfies (CC).




\section{Counting Gaussian rationals: the proof of Theorem \ref{thm: main cplx}}\label{sec: Guassian count}




\subsection{Preliminaries}\label{subsec: cplx prelim} Suppose $\cM \subset \C^{m+1}$ is a smooth, analytic hypersurface satisfying the complex curvature condition. By working locally, we may assume 
\begin{equation*}
    \cM = \{(\bdz, f(\bdz)) : \bdz \in U \}
\end{equation*}
where $U \subset \C^m$ is a bounded, open set and $f \colon U \to \C$ is a holomorphic function satisfying ($\C$-CC). By a quantitative version of the inverse function theorem (see, for example,  \cite[\S8]{Christ1985}), we may cover $U$ by small balls $B(\bdz_0, \tau^2)$ for $0 < \tau < 1$ such that
\begin{itemize}
    \item There exist constants $C_1$, $C_2 \geq 1$ depending only on $f$ such that
  \begin{multline}\label{eq: cplx ift 1}
      (\partial_{\bdz} f)\bigl(B(\bdz_0, \tau^2)\bigr) \subseteq B((\partial_{\bdz} f)(\bdz_0), C_1 \tau^2) \\
   \subseteq B((\partial_{\bdz} f)(\bdz_0), 2 C_1 \tau^2) \subseteq 
(\partial_{\bdz} f)\bigl(B(\bdz_0, C_2 \tau)\bigr);
\end{multline}
\item $\partial_{\bdz} f$ is a diffeomorphism between $B(\bdz_0, C_2 \tau)$ and
$(\partial_{\bdz} f)\bigl(B(\bdz_0, C_2 \tau)\bigr)$.
\end{itemize}




\subsubsection*{Smooth counting function} Using a smooth partition of unity, we can reduce matters to studying the following smooth variant of $N_{\cM}^{\C}(Q, \bddelta)$. Let $w \colon \C^m \to [0, 1]$ be smooth and supported in one of the balls $B(\bdz_0, \tau^2)$ as above. Given $Q \in \N$ and $\bddelta = (\delta_{\fkR}, \delta_{\fkI}) \in (0,1/4)^2$, we define 
\begin{equation*}
    \fkN_{\cM, w}^{\C}(Q, \bddelta) := \mathop{\sum_{q \in \Z[i],\, 0 < |q|_{\infty}\le Q} \ \sum_{\bda\in \Z[i]^m}}_{\|\fkR qf(\bda/q)\|\le \delta_{\fkR},\, \|\fkI qf(\bda/q)\| \le \delta_{\fkI}} w(\bda/q).
\end{equation*}
If $\bddelta = (\delta, \delta)$ for some $\delta \in (0, 1/4)$, then we simply write $\fkN_{\cM, w}^{\C}(Q, \delta)$ for $\fkN_{\cM, w}^{\C}(Q, \bddelta)$. Thus, to prove Theorem~\ref{thm: main cplx} for $m \geq 2$, our goal is to show
\begin{equation*}
\fkN_{\cM, w}^{\C}(Q, \delta)  \ll \delta^2 Q^{2m+1}  +  Q^{2m} \cE(Q),
\end{equation*}
with an appropriate modification in the $m=1$ case.




\subsubsection*{Duality} Define the quadratic form 
\begin{equation}\label{eq: quad form}
   \bdz \cdot \bdw := z_1 w_1 + \cdots + z_m w_m \qquad \textrm{for} \qquad \bdz = (z_1, \dots, z_m), \, \bdw = (w_1, \dots, w_m) \in \C^n. 
\end{equation}
Note, in particular, that $\bdz \cdot \bdw$ does not correspond to the usual hermitian inner product on $\C^n$. Recall that $\partial_{\bdz} f \colon \cD \to \cD^*$ is a diffeomorphism between 
\begin{equation*}\label{eq: cplx ift 2}
   \cD := B(\bdz_0, C_2 \tau) \qquad \textrm{and} \qquad \cD^* := (\partial_{\bdz} f)(B(\bdz_0, C_2 \tau)).
\end{equation*}
 The complex Legendre transform $f^{*} \colon \cD^* \to \C$ of $f$ is then defined by
\begin{equation*}
    f^{*}(\bdw) := - f(\bdz) + \bdz \cdot \bdw \qquad \textrm{where $\bdz \in \cD$ satisfies $\bdw = \partial_{\bdz} f(\bdz)$}
\end{equation*}
for $\bdw \in \cD^*$. It is easy to check that 
\begin{equation}\label{eq: cplx L trans}
   \partial_{\bdw} f^{*} (\bdw) = (\partial_{\bdz} f)^{-1}(\bdw), \ \ f^{**}(\bdz) = f(\bdz) \ \ {\rm and} \ \ 
(f^{*})''(\bdw) = (f'')^{-1} ((\partial_{\bdz} f)^{-1}(\bdw)). 
\end{equation}
Hence, by our hypothesis ($\C$-CC) on $f$, we see that $f^{*}$ also satisfies ($\C$-CC). We set
\begin{equation*}
    w^{*}(\bdw) := 
\frac{w((\partial_{\bdz} f)^{-1}(\bdw))}{|\det f''((\partial_{\bdz} f)^{-1}(\bdw))|^{1/2}}
\end{equation*}
so that $w^{**} = w$ from the properties \eqref{eq: cplx L trans} of the complex Legendre transform.

Given $Q_* \in \N$ and $\bddelta_* = (\delta_{\fkR}^*, \delta_{\fkI}^*) \in (0,1/4)^2$, we define the dual counting function
\begin{equation}\label{eq: dual count cplx}
  \fkN_{\cM^*, w^{*}}^{\C}(Q_*, \bddelta_*) := \mathop{\sum_{j \in \Z[i],\, 0 < |j|_{\infty}\le Q_*} \ \sum_{\bdk\in \Z[i]^m}}_{\|\fkR \overline{j} f(\overline{\bdk}/\,\overline{j}\,)\|\le \delta_{\fkR}^{*},\, \|\fkI \overline{j}f(\overline{\bdk}/\,\overline{j}\,)\| \le \delta_{\fkI}^*} w^{*}(\overline{\bdk}/\,\overline{j}\,).  
\end{equation}
Like in the real hypersurface case (and unlike the situation in \S\ref{sec: int dim}), the dual counting function $\fkN_{\cM^*, w^{*}}^{\C}(Q_*, \bddelta_*)$ has the same form as
$\fkN_{\cM, w}^{\C}(Q,\bddelta)$. 



\subsection{Self-improving mechanism}\label{subsec: cplx self improve}  The dual counting function $\fkN_{\cM^*, w^{*}}^{\C}$ is exactly the counting function $\fkN_{\cM, w}^{\C}$ with $f$ replaced by the Legendre dual $f^{*}$ (which also satisfies the curvature condition ($\C$-CC)) and $w$ replaced by $w^{*}$. Moreover, the dual counting function of the dual $\fkN_{\cM^*, w^{*}}^{\C}$ brings us back to the original counting function $\fkN_{\cM, w}^{\C}$. Hence, as in Huang's work on real hypersurfaces \cite{Huang2020}, the self-improving mechanism only involves a single lemma, the analogue of \cite[Lemma 4]{Huang2020}, showing that a bound for $\fkN_{\cM, w}^{\C}$ implies a bound for $\fkN_{\cM^*, w^{*}}^{\C}$. 

In order to state the key lemma, we observe that there exists some constant $A \geq 1$, depending only on $\cM$, such that
\begin{equation}\label{eq: choice of C0}
    \sum_{\substack{q \in \Z[i]\\ 0 < |q|_{\infty}\le Q}} \sum_{\bda\in \Z[i]^m} w(\bda/q) \leq (A/3\pi^4) Q^{2m+2}
\end{equation}
holds for all $Q \in \N$. Henceforth, $A$ denotes this fixed value.  

\begin{lemma}\label{lem: cplx self imp} Suppose that for all $\bddelta_* \in (0, 1/4)^2$ the bound
\begin{equation}\label{*bound}
\fkN_{\cM^*, w^{*}}^{\C}(Q_*, \bddelta_*) \le  A  \bddelta_*^{\times} Q_*^{2m+2} + B \, \bddelta_*^{\times}(\bddelta_*^{\wedge})^{-2} Q_*^{\beta} \log^{\gamma} Q_* 
\end{equation}
holds for some $2m < \beta \le 2m+2$, $\gamma \geq 0$ and $B \geq A \geq 1$. Then, for all $\bddelta \in (0, 1/4)^2$, the bound
\begin{equation}\label{w-bound}
\fkN_{\cM, w}^{\C}(Q,\bddelta) \le  A \bddelta^{\times} Q^{2m+2} +  C B \bddelta^{\times}(\bddelta^{\wedge})^{-2}
Q^{2m + 2 - \frac{2m}{\beta - m}} \log^{\frac{2(\gamma+1)}{m}} Q 
\end{equation}
holds. Here the constant $C \geq 1$ may depend on $\cM$ but is independent of the other parameters in the problem, including $B$. 
\end{lemma}

Here, similar to \S\ref{sec: int dim}, for $\bddelta = (\delta_{\fkR}, \delta_{\fkI}) \in (0,1/4)^2$, we write 
\begin{equation*}
    \bddelta^{\times} := \delta_{\fkR} \delta_{\fkI}, \qquad \bddelta^{\wedge} := \min\{\delta_{\fkR}, \delta_{\fkI}\} \qquad \textrm{and} \qquad \bddelta^{\vee} := \max\{\delta_{\fkR}, \delta_{\fkI}\}.
\end{equation*}
The above lemma can also be written in a dual form (see Lemma~\ref{lem: cplx self imp *} below), showing that a bound for $\fkN_{\cM, w}^{\C}$ implies a bound for $\fkN_{\cM^*, w^{*}}^{\C}$. Thus, once we have Lemma~\ref{lem: cplx self imp}, we can iteratively apply it and its dual form to arrive at Theorem~\ref{thm: main cplx}: see \S\ref{subsec: iter cplx} below. 

\begin{proof}[Proof (of Lemma~\ref{lem: cplx self imp})] Since the argument is somewhat lengthy, we break it into steps. \medskip

\noindent \textit{Step 1: Fourier reformulation}. Define $J_{\fkR} := \floor{\frac{1}{2\delta_{\fkR}}}$ and $J_{\fkI} := \floor{\frac{1}{2\delta_{\fkI}}}$, so that $J_{\fkR}^{-1}J_{\fkI}^{-1} \leq 16 \bddelta^{\times}$. Then, using the Fej\'er kernel as in \eqref{eq fejer char est}, together with the elementary identity
\begin{equation}\label{eq: cplx self imp 0}
  \fkR(\overline{j} z) = \fkR j \fkR z + \fkI j \fkI z \qquad \textrm{for all $j$, $z \in \C$,}
\end{equation}
 we have 
\begin{equation*}
  \fkN_{\cM, w}^{\C}(Q,\bddelta) \leq \bddelta^{\times} \sum_{\substack{j \in \Z[i] \\ |\fkR j| \leq J_{\fkR},\, |\fkI j| \leq J_{\fkI}}} a_{j, \bddelta} \sum_{\substack{q \in \Z[i]\\ 0 < |q|_{\infty}\le Q}} \ \sum_{\bda\in \Z[i]^m} w(\bda/q) e\big(\fkR(\overline{j}q f(\bda/q))\big),
\end{equation*}
where the real coefficients $a_{j, \bddelta}$ satisfy $a_{-j, \bddelta} = a_{j, \bddelta}$ and $0 \leq a_{j, \bddelta} \leq \pi^4$ for $j \in \Z[i]$. The right-hand side of the above inequality can be split into two terms
\begin{multline*}
   \pi^4\bddelta^{\times} \sum_{\substack{q \in \Z[i]\\ 0 < |q|_{\infty}\le Q}} \sum_{\bda\in \Z[i]^m} w(\bda/q) \\ 
   + \bddelta^{\times}  \sum_{\substack{j \in \Z[i] \setminus\{0\} \\ |\fkR j| \leq J_{\fkR} \\ |\fkI j| \leq J_{\fkI}}} a_{j, \bddelta} \sum_{\substack{q \in \Z[i]\\ 0 < |q|_{\infty}\le Q}} \ \sum_{\bda\in \Z[i]^m} w(\bda/q) e\big(\fkR(\overline{j}q f(\bda/q))\big). 
\end{multline*}
By the choice of $A \geq 1$ from \eqref{eq: choice of C0}, it follows that 
\begin{equation}\label{eq: cplx self imp 1}
    \fkN_{\cM, w}^{\C}(Q,\bddelta) \leq (A/3)\bddelta^{\times} Q^{2m+2} + O(\Sigma(Q, \bddelta))
\end{equation}
where
\begin{equation*}
  \Sigma(Q, \bddelta)  := \bddelta^{\times}\sum_{\substack{j \in \Z[i] \setminus\{0\} \\ |\fkR j| \leq J_{\fkR} \\ |\fkI j| \leq J_{\fkI}}} \Big| \sum_{\substack{q \in \Z[i]\\ 0 < |q|_{\infty}\le Q}} \ \sum_{\bda\in \Z[i]^m} w(\bda/q) e\big(\fkR(\overline{j}q f(\bda/q)\big)\Big|.
\end{equation*}
By the Poisson summation, together with a change of variables,\footnote{Here it is useful to pass to a function defined over $\R^{2m}$, writing the complex argument $\bdz = \bdx + i \bdy$ in real and imaginary parts. For $q = a + i b \neq 0$ with $a$, $b \in \Z$, the operation $\bdz \mapsto \bdz/q$ then corresponds to the matrix transformation
\begin{equation*}
\underbrace{[q] \otimes \cdots \otimes [q]}_{\textrm{$m$-fold}}  \qquad \textrm{where} \qquad [q] :=   \frac{1}{a^2 + b^2} \begin{pmatrix}
     a & b \\
     -b & a
    \end{pmatrix},    
\end{equation*}
with determinant $|q|^{2m}$, to which we apply the usual real change of variable formula.} we have 
\begin{equation*}
    \Sigma(Q, \bddelta)  = \bddelta^{\times} \sum_{\substack{j \in \Z[i] \setminus\{0\} \\ |\fkR j| \leq J_{\fkR} \\ |\fkI j| \leq J_{\fkI}}} \Big|
\sum_{\substack{q \in \Z[i]\\ 0 < |q|_{\infty}\le Q}} |q|^{2m} \sum_{\bdk\in \Z[i]^m}
\cI(j,q, \bdk)\Big|
\end{equation*}
where $|q|$ is the usual complex modulus and
\begin{equation*}
\cI(j,q, \bdk) := \int_{\C^m} w(\bdw) \, e\big(-\fkR( \overline{j} q \phi(\bdw; \overline{\bdk}/\overline{j}\,))\big) \ud \bdw \quad \textrm{and} \quad    \phi(\bdw; \bdz) :=  \bdw  \cdot \bdz - f(\bdw).
\end{equation*}
Here $\bdw \cdot \bdz$ is as defined in \eqref{eq: quad form} and we have again used \eqref{eq: cplx self imp 0}.\medskip

\noindent \textit{Step 2: Stationary phase}. We evaluate the $\cI(j,q, \bdk)$ using the standard asymptotic expansion for oscillatory integrals with nondegenerate phase (see, for instance, \cite[p.220, Theorem 7.7.5]{Hormander_book}). To do this, we consider two cases, defined in terms of the set $\cD^*$ from \eqref{eq: cplx ift 2}.\medskip

\noindent \underline{Case 1.} Non-stationary: $\overline{\bdk}/\overline{j} \notin \cD^*$.\medskip

\noindent In this case, it follows from the chain of inclusions in \eqref{eq: cplx ift 1} that $\overline{\bdk}/\overline{j} \notin B(\partial_{\bdz} f(\bdz_0), 2C_1 \tau^2)$ and so $|\overline{\bdk}/\overline{j} -\partial_{\bdz} f(\bdz_0)| \geq 2C_1 \tau^2$. On the other hand, for $\bdw \in \supp w \subseteq B(\bdz_0, \tau^2)$, it follows from the first inclusion in \eqref{eq: cplx ift 1} that $|\partial_{\bdz} f(\bdw) - \partial_{\bdz} f(\bdz_0)| < C_1 \tau^2$. Combining these observations with the triangle inequality,
\begin{equation}\label{eq: cplx self imp 1.1}
    |\partial_{\bdw} \phi(\bdw; \overline{\bdk}/\overline{j}\,)| \geq |\overline{\bdk}/\overline{j} -\partial_{\bdz} f(\bdz_0)| -  |\partial_{\bdz} f(\bdw) - \partial_{\bdz} f(\bdz_0)| \geq C_1 \tau^2 \gg 1. 
\end{equation}
We now write $\phi$ in real and imaginary parts, so that 
\begin{equation*}
   \phi(\bdx + i \bdy; \overline{\bdk}/\overline{j}\,) = u_{\phi}(\bdx, \bdy; \overline{\bdk}/\overline{j}\,) + i v_{\phi}(\bdx, \bdy; \overline{\bdk}/\overline{j}\,)
\end{equation*}
where $u_{\phi}$ and $v_{\phi}$ are real-valued. Using \eqref{eq: cplx self imp 0}, our phase is then given by
\begin{equation}\label{eq: cplx self imp 1.2}
   \fkR( \overline{j} q \phi(\bdx + i \bdy; \overline{\bdk}/\overline{j}\,)) =  \fkR(j \overline{q}) u_{\phi}(\bdx, \bdy; \overline{\bdk}/\overline{j}\,) + \fkI(j \overline{q}) v_{\phi}(\bdx, \bdy; \overline{\bdk}/\overline{j}\,).
\end{equation}
Using the Cauchy--Riemann equations, the $(\bdx, \bdy)$-gradient of the phase function satisfies
\begin{equation}\label{eq: cplx self imp 1.3}
\nabla_{(\bdx, \bdy)} \fkR\Big( \frac{\overline{j} q}{|q||j|} \phi(\bdx + i \bdy; \overline{\bdk}/\overline{j}\,)\Big) 
=
\frac{1}{|q||j|} 
\begin{pmatrix}
    \fkR( \overline{j} q) I_m & \fkI( \overline{j} q)I_m \\
    \fkI( \overline{j} q) I_m & -\fkR( \overline{j} q)I_m
\end{pmatrix}
\begin{pmatrix}
    \nabla_{\bdx} u_{\phi}(\bdx, \bdy; \overline{\bdk}/\overline{j}\,) \\
    \nabla_{\bdx} v_{\phi}(\bdx, \bdy; \overline{\bdk}/\overline{j}\,)
\end{pmatrix}.
\end{equation}
The matrix appearing on the right-hand side of the above display belongs to $\mathrm{SO}(2m)$. Hence, combining \eqref{eq: cplx self imp 1.1} and \eqref{eq: cplx self imp 1.3}, we deduce that
\begin{equation*}
    |\nabla_{(\bdx, \bdy)} \fkR( \overline{j} q \phi(\bdx + i \bdy; \overline{\bdk}/\overline{j}\,))| \gg |q||j| \qquad \textrm{for all $\bdx$, $\bdy \in \R^m$ with $\bdx + i \bdy \in \supp w$.}
\end{equation*}
On the other hand, if $|\bdk| \geq C|j|$ for a suitably large constant $C$, then it is clear that 
\begin{equation*}
    |\nabla_{(\bdx, \bdy)} \fkR( \overline{j} q \phi(\bdx + i \bdy; \overline{\bdk}/\overline{j}\,))| \gg |\bdk|.
\end{equation*}
Thus, by repeated integration-by-parts,
\begin{equation*}
    |\cI(j,q, \bdk)| \ll_{M, N} |q|^{-N}|j|^{-N}|\bdk/j|^{-M} \qquad \textrm{for all $M$, $N \in \N_0$}
\end{equation*}
and the contribution in this case is essentially negligible. \medskip 

\noindent \underline{Case 2.} Stationary: $\overline{\bdk}/\overline{j} \in \cD^*$.\medskip

Since $\supp w \subseteq B(\bdz_0, \tau^2) \subseteq \cD$ (where the latter inclusion holds by \eqref{eq: cplx ift 1}), it follows that there exists a unique critical point $\bdz_0 \in \supp w$ such that $\nabla_{\bdy}\phi(\bdz_0, \overline{\bdk}/\overline{j}) = 0$. In particular, $\bdz_0 = (\partial_{\bdz}f)^{-1}(\overline{\bdk}/\overline{j})$ and, by the definition of the Legendre dual, 
\begin{equation*}
    \phi(\bdz_0, \overline{\bdk}/\overline{j}) = f^*(\overline{\bdk}/\overline{j}). 
\end{equation*}

The condition ($\C$-CC) implies that the critical point $\bdz_0$ is nondegenerate over $\C$. Moreover, we can use Lemma~\ref{lem: ND vs CND} to translate this into a nondegeneracy condition over $\R$. Indeed, taking the hessian of both sides of \eqref{eq: cplx self imp 1.2}, we deduce that 
\begin{equation*}
    \Big|\det \partial_{(\bdx, \bdy), (\bdx, \bdy)}^2 \fkR\Big( \frac{\overline{j} q}{|q||j|} \phi(\bdx + i \bdy; \overline{\bdk}/\overline{j}\,)\Big)\Big| 
=
\big|\det
\begin{pmatrix}
    \theta_1 u''(\bdx, \bdy) + \theta_2 v''(\bdx, \bdy)
\end{pmatrix}\big|  
\end{equation*}
where $u$ and $v$ denote the real and imaginary parts of $f$ and $(\theta_1, \theta_2) \in S^1$ is a unit vector parallel to $(\fkR(j \overline{q}), \fkI(j \overline{q}))$. However, By Lemma~\ref{lem: ND vs CND} and the hypothesis that $f$ satisfies ($\C$-CC), the right-hand determinant is non-vanishing over its entire domain. 

In light of the above, we may appeal to stationary phase asymptotics for oscillatory integrals with nondegenerate phase (see, for instance, \cite[p.220, Theorem 7.7.5]{Hormander_book}), we obtain
\begin{equation}\label{eq: cplx self imp 1.5}
    \Sigma(Q, \bddelta) = \Sigma_0(Q, \bddelta) + E(Q, \bddelta) \quad \textrm{where} \quad E(Q, \bddelta) = O\big((\bddelta^{\wedge})^{-(m-1)} Q^{m+1}\big),
\end{equation}
where
\begin{equation}\label{eq: cplx self imp 2}
 \Sigma_0(Q, \bddelta) := \bddelta^{\times}  \sum_{\substack{j \in \Z[i] \setminus\{0\} \\ |\fkR j| \leq J_{\fkR} \\ |\fkI j| \leq J_{\fkI}}} \sum_{\substack{q \in \Z[i] \\0 < |q|_{\infty}\le Q}} |q|^{2m} \sum_{\substack{\bdk \in \Z[i]^m \\ \overline{\bdk}/\overline{j} \in \cD^*}} \frac{w^{*}(\overline{\bdk}/\,\overline{j}\,)}{|\overline{j} q|^m} e\big(-\fkR(\overline{j}q f^{*}(\overline{\bdk}/\,\overline{j}\,))\big).   
\end{equation}
The next step is to relate the exponential sums appearing in \eqref{eq: cplx self imp 2} to the dual counting function appearing in our hypothesis \eqref{*bound}.\medskip

\noindent \textit{Step 3: Passing to the dual counting problem}. We estimate the $q$ sum in \eqref{eq: cplx self imp 2} using summation-by-parts as in Lemma~\ref{lem: summation by parts}. This gives
\begin{equation}\label{eq: cplx self imp 3}
    \Bigl| \sum_{\substack{q \in \Z[i] \\0 < |q|_{\infty}\le Q}} |q|^m  e\big(-\fkR (\overline{j} qf^{*}(\overline{\bdk}/\,\overline{j}\,))\big) \Bigr|  \ll  Q^m \prod_{\fkX \in \{\fkR, \fkI\}} \min\Big\{Q, \frac{1}{\|\fkX \overline{j}f^*(\overline{\bdk}/ \overline{j}\,)\|}\Big\}.
\end{equation}
Combining \eqref{eq: cplx self imp 2} and \eqref{eq: cplx self imp 3}, we have
\begin{align}\label{eq: cplx self imp 4}
 \Sigma_0(Q, \bddelta) \ll \bddelta^{\times} Q^m  \sum_{\substack{j \in \Z[i] \setminus\{0\} \\ |\fkR j| \leq J_{\fkR} \\ |\fkI j| \leq J_{\fkI}}} |j|_{\infty}^{-m} \sum_{\substack{\bdk \in \Z[i]^m \\ \overline{\bdk}/\overline{j} \in \cD^*}} w^*(\overline{\bdk}/\overline{j}\,) \prod_{\fkX \in \{\fkR, \fkI\}} \min \Big\{Q, \frac{1}{\|\fkX \overline{j} f^*(\overline{\bdk}/ \overline{j}\,)\|}\Big\}.
\end{align}
We split the sum in \eqref{eq: cplx self imp 4} according to the relative sizes of $\|\fkR qf^*(\overline{\bdk}/ \overline{j}\,)\|$, $\|\fkI qf^*(\overline{\bdk}/ \overline{j}\,)\|$ and $Q$. Setting $ Q_* := \floor{\frac{1}{2\bddelta^{\wedge}}}$, this leads to the bound
\begin{equation}\label{eq: cplx self imp 5}
  \Sigma_0(Q, \bddelta) \ll \bddelta^{\times} Q^{m+2} \sum_{\substack{\bdi = (i_{\fkR}, i_{\fkI}) \in \N_0^2 \\ 1 \leq 2^{i_{\fkR}}, 2^{i_{\fkI}} \leq Q}} 2^{-i_{\fkR} - i_{\fkI}}  \Big[\mathop{\sum_{\substack{j \in \Z[i] \\ 0 < |j|_{\infty} \le Q_*}} |j|_{\infty}^{-m}\sum_{\substack{\bdk \in \Z[i]^m \\ \overline{\bdk}/\overline{j} \in \cD^*}}}_{\substack{\|\fkR \overline{j} f^*(\overline{\bdk}/\overline{j}\,)\|\le 2^{i_{\fkR}} Q^{-1} \\ \|\fkI \overline{j} f^*(\overline{\bdk}/\overline{j}\,)\|\le 2^{i_{\fkI}} Q^{-1}}} w^{*}(\overline{\bdk}/\,\overline{j}\,)\Big],  
\end{equation}
where we have relaxed the range of summation in $j$ in order to later apply \eqref{*bound}. Applying a dyadic decomposition and the definition of the dual counting function from \eqref{eq: dual count cplx}, the bracketed expression on the right-hand side of \eqref{eq: cplx self imp 5} may be bounded by
\begin{equation}\label{eq: cplx self imp 6}
\sigma_{\bdi}(Q; \bddelta) := \sum_{\substack{1 \leq \lambda \leq Q_* \\ \mathrm{dyadic}}} \lambda^{-m} \cdot \fkN_{\cM^*, w^*}^{\C}(\lambda, \bddelta_{*, \bdi}),
\end{equation}
where $\bddelta_{*, \bdi} = (2^{i_{\fkR}}Q^{-1}, 2^{i_{\fkI}}Q^{-1})$ for $\bdi = (i_{\fkR}, i_{\fkI}) \in \N_0^2$ satisfying $1 \leq 2^{i_{\fkR}}, 2^{i_{\fkI}} \leq Q$.\medskip

\noindent \textit{Step 4: Applying the hypothesised bounds}. Observe that
\begin{equation*}
    \bddelta_{*, \bdi}^{\times} = 2^{i_{\fkR} + i_{\fkI}} Q^{-2} \qquad \textrm{and} \qquad \bddelta_{*, \bdi}^{\times}(\bddelta_{*, \bdi}^{\wedge})^{-2} = 2^{i_{\fkR}+ i_{\fkI} - 2\min\{i_{\fkR}, i_{\fkI}\}}.
\end{equation*}
Recalling $Q_* \sim (\bddelta^{\wedge})^{-1}$, we now apply the hypothesised bound \eqref{*bound} to estimate each term in \eqref{eq: cplx self imp 6}, giving 
\begin{equation*}
   \sigma_{\bdi}(Q; \bddelta) \ll  B2^{i_{\fkR} + i_{\fkI}}\big(  Q^{-2} (\bddelta^{\wedge})^{-m - 2} +2^{-2 \min\{i_{\fkR}, i_{\fkI}\}} (\bddelta^{\wedge})^{
-(\beta-m)} |\log \bddelta^{\wedge}|^{\gamma}\big). 
\end{equation*}
Here we have used the hypothesis $\beta > 2m > m$. Applying the above bound to the right-hand side of \eqref{eq: cplx self imp 5} and summing over $\bdi = (i_{\fkR}, i_{\fkI}) \in \N_0^2$ satisfying $1 \leq 2^{i_{\fkR}}, 2^{i_{\fkI}} \leq Q$, we deduce that 
\begin{equation}\label{eq: cplx self imp 7}
    \Sigma_0(Q, \bddelta) \ll \bddelta^{\times}(\bddelta^{\wedge})^{-m - 2} Q^m \log^2 Q + \bddelta^{\times}(\bddelta^{\wedge})^{-(\beta-m)} |\log \bddelta^{\wedge}|^{\gamma} Q^{m+2} \log Q.
\end{equation}

We compare the size of the right-hand side of \eqref{eq: cplx self imp 7} to that of the error term $E(Q, \bddelta) = O\big((\bddelta^{\wedge})^{-(m-1)} Q^{m+1}\big)$ from \eqref{eq: cplx self imp 1.5}.
\begin{itemize}
    \item If $(\bddelta^{\wedge}) Q \leq 1$, then $E(Q, \bddelta)$ is controlled by the first term on the right-hand side of \eqref{eq: cplx self imp 7}; 
    \item If $(\bddelta^{\wedge}) Q \geq 1$, then $E(Q, \bddelta)$ is controlled by the second term on the right-hand side of \eqref{eq: cplx self imp 7}.
\end{itemize}
In the second case, we again used the hypothesis $\beta > 2m > m$. 
Combining the above observations with \eqref{eq: cplx self imp 1}, \eqref{eq: cplx self imp 1.5} and \eqref{eq: cplx self imp 7}, we therefore obtain 
\begin{align}\label{eq: cplx self imp 7.5}
   & \fkN_{\cM, w}^{\C}(Q,\bddelta)  \leq \\
   \nonumber
   & \underbrace{(A/3)\bddelta^{\times} Q^{2m+2}}_{\tO \,:=} + \underbrace{CB\bddelta^{\times}(\bddelta^{\wedge})^{-m - 2} Q^m \log^2 Q}_{\tI \,:=} + \underbrace{CB\bddelta^{\times}(\bddelta^{\wedge})^{-(\beta-m)} |\log \bddelta^{\wedge}|^{\gamma} Q^{m+2} \log Q}_{\tII\,:=}.
\end{align}
It remains to compare the terms appearing in the above bound in order to arrive at our final estimate.\medskip

\noindent \textit{Step 5: Comparing terms}.  For $c_0 :=(3CB)/A$, with $C \geq 1$ as in \eqref{eq: cplx self imp 7.5}, we make the following observations:
\begin{equation*}
 \tO \le \tI \quad \Longleftrightarrow \quad \bddelta^{\wedge} \le \delta_{\tI}  := c_0^{1/(m+2)} Q^{- 1} \, \log^{2/(m+2)} Q, 
\end{equation*}
and
\begin{equation*}
    \tO \le \ \tII  \quad \Longleftrightarrow \quad \bddelta^{\wedge} \le \delta_{\tII} \qquad \textrm{where} \qquad
\delta_{\tII}^{\beta - m}|\log \delta_{\tII}|^{-\gamma} = c_0 Q^{-m}  \log Q.
\end{equation*}
The latter condition implies that $|\log \delta_{\tII}| \sim \log Q$. Hence
\begin{equation}\label{eq: cplx self imp 8}
  \delta_{\tII} \ll (B/A)^{1/2} Q^{-m/(\beta - m)} \log^{(\gamma+2)/m}Q, 
\end{equation}
where we have bounded $(B/A)^{1/(\beta-m)} \leq (B/A)^{1/2}$ and $(\gamma + 1)/(\beta - m) \leq (\gamma + 2)/m$, using $\beta \geq 2m \geq 2$, $\gamma \geq 0$ and $B \geq A$. Furthermore, we see that $\delta_{\tI} \ll \delta_{\tII}$ since $2m < \beta \leq 2m + 2$.\medskip

If $\bddelta^{\wedge} \geq \max\{\delta_{\tI}, \delta_{\tII}\}$, then the above observations imply $\tO \geq \max \{\tI, \tII\}$ and so
\begin{equation}\label{eq: cplx self imp 9}
   \fkN_{\cM, w}^{\C}(Q,\bddelta)  \leq A\bddelta^{\times} Q^{2m+2},
\end{equation}
which is a favourable estimate. On the other hand, if $\bddelta^{\wedge} \leq \max\{\delta_{\tI}, \delta_{\tII}\}\leq \bddelta^{\vee}$, then we can apply the monotonicity of the counting function to deduce that 
\begin{equation*}
\fkN_{\cM, w}^{\C}(Q,\bddelta) \leq \fkN_{\cM, w}^{\C}(Q,\bdeta) \quad \textrm{where} \quad \bdeta^{\wedge} = \max\{\delta_{\tI}, \delta_{\tII}\} \quad \textrm{and}  \quad \bdeta^{\vee} = \bddelta^{\vee}.
\end{equation*}
Observe that
\begin{align*}
    \bdeta^{\times} &= \bddelta^{\times} (\bddelta^{\wedge})^{-2} \bddelta^{\wedge}  \max\{\delta_{\tI}, \delta_{\tII}\} \\ 
    &\leq \bddelta^{\times} (\bddelta^{\wedge})^{-2} \max\{\delta_{\tI}, \delta_{\tII}\}^2 \\
    &\ll (B/A) \bddelta^{\times} (\bddelta^{\wedge})^{-2} Q^{-2m/(\beta - m)} \log^{2(\gamma+2)/m} Q,
\end{align*}
where the last step follows from the fact that $\max\{\delta_{\tI}, \delta_{\tII}\} \ll \delta_{\tII}$ and \eqref{eq: cplx self imp 8}. Since $\bdeta^{\wedge} = \max\{\delta_{\tI}, \delta_{\tII}\}$, the estimate \eqref{eq: cplx self imp 9} from the previous case applies to $\fkN_{\cM, w}^{\C}(Q,\bdeta)$, and we therefore deduce that
\begin{equation*}
 \fkN_{\cM, w}^{\C}(Q,\bddelta)  \leq A \bdeta^{\times} Q^{2m+2}  \ll B\bddelta^{\times} (\bddelta^{\wedge})^{-2}  Q^{2m+2-2m/(\beta - m)} \log^{2(\gamma+2)/m} Q.
\end{equation*}
Finally, if $\bddelta^{\wedge} \leq \bddelta^{\vee} \leq \max\{\delta_{\tI}, \delta_{\tII}\}$, then a minor adaption of the above argument leads to the same conclusion, this time taking $\bdeta^{\wedge} = \bdeta^{\vee} = \max\{\delta_{\tI}, \delta_{\tII}\}$. Thus, in all cases we obtain a desirable bound. 
\end{proof}




\subsection{Iteration scheme}\label{subsec: iter cplx} Here we show that repeated application of the self-improving mechanism from Lemma~\ref{lem: cplx self imp} leads to Theorem~\ref{thm: main cplx}. It is important to note that, in order to establish the bound \eqref{w-bound} even for the case $\delta_{\fkR} = \delta_{\fkI}$, we need the bound \eqref{*bound} for all non-isotropic $\bddelta = (\delta_{\fkR}, \delta_{\fkI}) \in (0,1/4)^2$.

As discussed at the beginning of \S\ref{subsec: cplx self improve}, the self improving mechanism from Lemma~\ref{lem: cplx self imp} has an equivalent dual formulation.

\begin{lemmabis}{lem: cplx self imp}\label{lem: cplx self imp *}
Suppose for all for all $\bddelta \in (0, 1/4)^2$ the bound
\begin{equation}\label{w-bound-gamma}
\fkN_{\cM, w}^{\C}(Q,\bddelta) \le  A \bddelta^{\times} Q^{2m+2} + B \bddelta^{\times} (\bddelta^{\wedge})^{-2} Q^{\beta} \log^{\gamma} Q 
\end{equation}
holds for some $2m < \beta \leq 2m+2$, $\gamma \geq 0$ and $B \geq A \geq 1$. Then, for all $\bddelta_* \in (0, 1/4)^2$, the bound
the bound
\begin{equation*}
\fkN_{\cM^*, w^{*}}^{\C}(Q_*, \bddelta_*)  \leq A \bddelta_*^{\times}Q_*^{2m+2} +  CB  \bddelta_*^{\times} (\bddelta_*^{\wedge})^{-2} Q_*^{2m + 2 - \frac{2m}{\beta - m}} \log^{\frac{2(\gamma+2)}{m}} Q_* 
\end{equation*}
holds. Here the constant $C \geq 1$ may depend on $\cM$ but is independent of the other parameters in the problem, including $B$. 
\end{lemmabis}

One may derive Lemma~\ref{lem: cplx self imp *} as an immediate consequence of Lemma~\ref{lem: cplx self imp} using the duality between the counting functions $\fkN_{\cM, w}^{\C}(Q,\bddelta)$ and $\fkN_{\cM^*, w^{*}}^{\C}(Q_*, \bddelta_*) $, as discussed in \S\S\ref{subsec: cplx prelim}-\ref{subsec: cplx self improve}. With this in hand, we may turn to the proof of Theorem~\ref{thm: main cplx}.

\begin{proof}[Proof (of Theorem~\ref{thm: main cplx})] Provided $B \geq A \geq 1$ are suitably chosen, either bound \eqref{*bound} or \eqref{w-bound-gamma} trivially holds for $\beta = 2m+2$ and $\gamma = 0$. We may then alternatively apply Lemma~\ref{lem: cplx self imp} and Lemma~\ref{lem: cplx self imp *} to deduce better and better bounds for $\fkN_{\cM, w}^{\C}(Q,\bddelta)$ and $\fkN_{\cM^*, w^{*}}^{\C}(Q_*, \bddelta_*) $. In particular, if we recursively define
\begin{align*}
    \beta_0 &:= 2m+2,& \qquad \qquad \gamma_0 &:= 0, & \\
    \beta_{j+1} &:= 2m + 2 - \frac{2m}{\beta_j - m}, & \qquad \qquad \gamma_{j+1} &:= \frac{2(\gamma_j+2)}{m} & \textrm{for $j \in \N_0$,}
\end{align*}
then a simple inductive argument shows that $2m < \beta_{j+1} < \beta_j \leq 2m+2$ and $\gamma_j \geq 0$ for all $j \in \N_0$. Thus, from the argument described above, we conclude that there exists some constant $C \geq 1$ such that 
\begin{equation}\label{eq: cplx pf 1}
\fkN_{\cM, w}^{\C}(Q,\bddelta) \le  A \bddelta^{\times} Q^{2m+2} + C^j \bddelta^{\times} (\bddelta^{\wedge})^{-2} Q^{\beta_j} \log^{\gamma_j} Q \qquad \textrm{holds for all $j \in \N_0$.} 
\end{equation}
The analysis now splits into three cases depending on the dimension.\medskip

\noindent\underline{Case: $m \geq 3$}. For $j \in \N$, an inductive argument shows $0 \leq \gamma_j \leq 4$. On the other hand,
\begin{equation*}
    \beta_j - 2m = \frac{2(\beta_{j-1} - 2m)}{\beta_{j-1} - m} \qquad \textrm{and so} \qquad |\beta_{j+1} - 2m| \leq (2/m) |\beta_{j-1} - 2m|,
\end{equation*}
since $\beta_{j-1} \geq 2$. Repeatedly applying this inequality, 
\begin{equation*}
   |\beta_j - 2m| \leq  (2/m)^j |\beta_0 - 2m| = 2(2/m)^j \qquad \textrm{for all $j \in \N_0$.}
\end{equation*}
To optimise our estimate, we choose $j \sim \log\log Q$ in \eqref{eq: cplx pf 1} and thereby obtain
\begin{equation*}
    \fkN_{\cM, w}^{\C}(Q,\bddelta) \ll \delta^2 Q^{2m+2} + Q^{2m} \log^{\kappa_m} Q
\end{equation*}
for some dimensional constant $\kappa_m \geq 1$.\medskip

\noindent\underline{Case: $m = 2$}. In this case, one can show by induction that the terms of the sequences $(\beta)_{j \in \N_0}$ and $(\gamma_j)_{j \in \N_0}$ are given by 
\begin{equation*}
    \beta_j = 4 + \frac{2}{j+1} \qquad \textrm{and} \qquad \gamma_j := 2j \qquad \textrm{for all $j \in \N_0$.}
\end{equation*}
Thus, the sequence is converging to the target value $4$, but at a much slower rate when compared with the $m \geq 3$ case. To optimise our estimate, we choose $j \sim \sqrt{\log Q}/\sqrt{\log\log Q}$ in \eqref{eq: cplx pf 1} and thereby obtain 
\begin{equation*}
    \fkN_{\cM, w}^{\C}(Q,\delta) \ll \delta^2 Q^6 + Q^4 e^{\kappa_2 \sqrt{\log Q \log\log Q}}.
\end{equation*}
for some absolute constant $\kappa_2 \geq 1$. This is analogous to the case of $2$-dimensional real hypersurfaces in $\R^3$ in the work of Huang~\cite{Huang2020}. \medskip

\noindent\underline{Case: $m = 1$}. In this case, $\beta_0 = 4$, $\beta_1 = 3$, $\gamma_0 = 0$ and $\gamma_1 = 2$. However, $3$ is a fixed point of the recursion relation defining $(\beta_j)_{j \in \N_0}$ and so
\begin{equation*}
    \fkN_{\cM, w}^{\C}(Q,\delta) \ll \delta^2 Q^4 + Q^3 \log^4 Q,
\end{equation*}
corresponding to $j = 1$, is the best bound attainable using the above method. This is analogous to case of $1$-dimensional real hypersurfaces in $\R^2$ (that is, plane curves) in the work of Huang~\cite{Huang2020}. 
\end{proof}

\begin{remark} Modifying the above argument easily yields an improved $\log$ power in the $m = 1$ case. However, for presentational convenience, we chose not to optimise here, since in any case the $Q^3$ term is likely suboptimal.     
\end{remark}




\appendix

\section{Summation-by-parts}

Here, for completeness, we provide the simple and standard proof of Lemma~\ref{lem: summation by parts}.

\begin{proof}[Proof (of Lemma~\ref{lem: summation by parts})] We induct on $n \in \N_0$. It is convenient to take the base case to be $n = 0$, in which case $\fka$ is interpreted as a constant and the condition $\fka \in \fkA(\lambda;\bdJ; B)$ is interpreted as $|\fka| \leq B$. Furthermore, the exponential sum $S_{\fka}(x, \bdJ)$ is interpreted as equal to the constant $\fka$ and so the desired inequality is just a restatement of the hypothesis.\medskip

Now suppose the result holds for some $n \in \N_0$. 
Let $\lambda \geq 1$ and $\bdJ \in \fkJ^{n+1}(\lambda)$ be given by $\bdJ = \bdI \times J$ where $\bdI \in \fkJ^n(\lambda)$ and $J := [M, N]$ for $M$, $N \in \N_0$ with $0 \leq M < N \leq \lambda$. Fix $B > 0$ and for $\fka \in \fkA(\lambda; \bdJ; B)$ and $\bdz = (\bdx, y) \in \R^n \times \R$ consider the sum 
\begin{equation*}
    S_{\fka}(\bdz, \bdJ) = \sum_{\ell = M}^N S_{\fka(\,\cdot\,, \ell)}(\bdx, \bdI)  e(\ell y)
\end{equation*}
where
\begin{equation*}
    S_{\fka(\,\cdot\,, \ell)}(\bdx, \bdI) = \sum_{\bdj \in \bdI \cap \N_0^n} \fka(\bdj, \ell) e(\bdj \cdot \bdx).
\end{equation*}
For $k \in \N$, we define $S(y, k) := \sum_{j = 1}^k e(j y)$. By the familiar summation-by-parts formula,
\begin{equation}\label{eq: mulit w lin est 1}
   |S_{\fka}(\bdz, \bdJ)| \leq  |S_{\fka(\,\cdot\,, N)}(\bdx, \bdI)| |S(y, N)|  + \sum_{k = M}^{N-1}  | S_{(\Delta\fka)(\,\cdot\,, k)}(\bdx, \bdI)| |S(y, k)|,
\end{equation}
where here $(\Delta\fka)(\,\cdot\,, k) := \fka(\,\cdot\,, k+1) - \fka(\,\cdot\,, k)$.

We now apply the induction hypothesis to bound
\begin{equation}\label{eq: mulit w lin est 2}
  |S_{\fka(\,\cdot\,, N)}(\bdx, \bdI)| \leq B_0(N)  \prod_{\ell = 1}^n \min\Big\{\lambda, \frac{1}{\|x_{\ell}\|}\Big\}
\end{equation}
where
\begin{equation}\label{eq: mulit w lin est 3}
B_0(N) := \sum_{\alpha \in \{0,1\}^n} \lambda^{|\alpha|} \sup_{\bdu \in \bdI} |\partial_{\bdu}^{\alpha}\fka(\bdu, N)|  \leq  \sum_{\substack{\alpha \in \{0,1\}^{n+1} \\ \alpha_{n+1} = 0}} \lambda^{|\alpha|} \sup_{\bdu \in \bdJ} |\partial_{\bdu}^{\alpha}\fka(\bdu)|.
\end{equation}
Similarly, for $M \leq k \leq N-1$, the induction hypothesis also implies 
\begin{equation}\label{eq: mulit w lin est 4}
    |S_{(\Delta\fka)(\,\cdot\,, k)}(\bdx, \bdI)|\leq B_1(k)  \prod_{\ell = 1}^n \min\Big\{\lambda, \frac{1}{\|x_{\ell}\|}\Big\}
\end{equation}
where, by the mean value theorem,
\begin{equation}\label{eq: mulit w lin est 5}
B_1(k) := \sum_{\alpha \in \{0,1\}^n} \lambda^{|\alpha|} \sup_{\bdu \in \bdI} |\partial_{\bdu}^{\alpha}\Delta\fka(\bdu, k)|  \leq  \lambda^{-1}\sum_{\substack{\alpha \in \{0,1\}^{n+1} \\ \alpha_{n+1} = 1}} \lambda^{|\alpha|} \sup_{\bdu \in \bdJ} |\partial_{\bdu}^{\alpha}\fka(\bdu)|.
\end{equation}
Applying the inequalities \eqref{eq: mulit w lin est 2} and \eqref{eq: mulit w lin est 4} to the right-hand side of \eqref{eq: mulit w lin est 1} and recalling that $S(y, k) := \sum_{j = 1}^k e(j y)$ satisfies
\begin{equation*}
    |S(y, k)| \leq \min\Big\{k, \frac{1}{\|y\|}\Big\},
\end{equation*}
we deduce that
\begin{equation}\label{eq: mulit w lin est 6}
    |S_{\fka}(\bdz, \bdJ)| \leq  \Big(B_0(N)  + \sum_{k = M}^{N-1}  B_1(k) \Big) \cdot \prod_{\ell = 1}^n \min\Big\{\lambda, \frac{1}{\|x_{\ell}\|}\Big\} \cdot \min\Big\{\lambda, \frac{1}{\|y\|}\Big\}.
\end{equation}
Finally, applying \eqref{eq: mulit w lin est 3} and \eqref{eq: mulit w lin est 5} to estimate the above display and using the hypothesis $\fka \in \fkA(\lambda; \bdJ;B)$, we conclude that
\begin{equation*}
    |S_{\fka}(\bdz, \bdJ)| \leq  B \prod_{\ell = 1}^{n+1} \min\Big\{\lambda, \frac{1}{\|z_{\ell}\|}\Big\}.
\end{equation*}
Here, the number of terms in the sum in $k$ in \eqref{eq: mulit w lin est 6} is $N - M \leq \lambda$, which cancels with the $\lambda^{-1}$ factor from \eqref{eq: mulit w lin est 5}. This closes the induction and completes the proof.
\end{proof}



\bibliography{Reference}
\bibliographystyle{amsplain}

\end{document}